\documentclass[11pt,oneside]{amsart}
\usepackage[top=25mm, bottom=25mm, left=20mm, right=20mm]{geometry}
\usepackage{amsfonts,amssymb,amsmath,amsthm,amsxtra}
\usepackage[T1]{fontenc}
\usepackage[utf8]{inputenc}
\usepackage{bm}
\usepackage{mathtools}
\usepackage{dsfont}
\usepackage{mathrsfs}
\usepackage{enumitem}
\allowdisplaybreaks

\usepackage{graphics}
\usepackage{hyperref}

\numberwithin{equation}{section}
\newtheorem{theorem}[equation]{Theorem}

\newtheorem{lemma}[equation]{Lemma}
\newtheorem{proposition}[equation]{Proposition}
\theoremstyle{definition}

\newtheorem{definition}[equation]{Definition}

\newcommand{\Mod}[1]{\ (\mathrm{mod}\ #1)}

\DeclareMathOperator\Lip{Lip}
\author{Leonidas Daskalakis} 
\thanks{The author was supported by the National Science Center (NCN Poland) Sonata grant No. 2024/55/D/ST1/00468.}
\address[Leonidas Daskalakis]{Jagiellonian University, Faculty of Mathematics and Computer Science, Institute of Mathematics, 30-348 Krak\'ow, Poland}
\email{leonidas.daskalakis@uj.edu.pl}
\begin{document}
\title[Pointwise ergodic theorems along oscillating sequences]{Pointwise ergodic theorems\\ along oscillating sequences}
\maketitle
\begin{abstract} We establish the first pointwise convergence result for nonconventional ergodic averages modeled along a large class of oscillating sequences, treating, for example, orbits of the form $\big(\lfloor n^d\sin(e\pi n)\rfloor\big)_{n\in\mathbb{N}}$, $d\in\mathbb{N}$. More specifically, for every probability space $(X,\mathcal{B},\mu)$ equipped with an invertible measure-preserving transformation $T\colon X\to X$ and every $f\in L^p_{\mu}(X)$ with $p\in(1,\infty]$, we prove that the limit
\[
\lim_{N\to\infty}\frac{1}{N}\sum_{n=1}^Nf\big(T^{\lfloor P(n)\psi(\alpha n)\rfloor}x\big)\quad\text{exists for $\mu$-a.e. $x\in X$,}
\]
where $P$ is a non-constant polynomial with real coefficients, $\psi$ is $1$-periodic and real-analytic, and $\alpha$  has irrationality exponent smaller than $1+\sqrt{\deg(P)+1}$. On the one hand, the latter assumption is satisfied by a.e. $\alpha\in\mathbb{R}$, and, on the other hand, our result covers any of the following concrete orbits
\[
\lfloor n\cos(e\pi n)\rfloor\text{,}\quad\lfloor n^{37}\sin(n)\rfloor\text{,}\quad\big\lfloor \sqrt{2}n^{10}e^{\sin(\pi (\sqrt[5]{7}+1) n)}\big\rfloor\text{,}\quad \big\lfloor \big((\log 2) n^3+e n\big)e^{e^{\cos(\pi \sqrt{2} n)}}\big\rfloor\text{.}
\]
The proof employs the so-called one-frequency circle method for estimating exponential sums with phases $\big(\xi \lfloor P(n)\psi(\alpha n)\rfloor\big)_{1\le n\le N}$, but the major arc approximants of our exponential sums do not consist of their direct continuous counterparts. In the major arc regime we decouple the $\psi$-oscillatory part of the phase from the polynomial one, and approximate the exponential sum by $\int_0^1\big(\frac{1}{N}\int_0^Ne^{2\pi i\xi P(t)\psi( s)}dt\big)ds$. This represents a novel phenomenon for problems of this kind, and, interestingly, our proof exploits the coefficient-free variational estimates of Radon-type operators.
\end{abstract}
\section{Introduction}
\subsection{Brief historical remarks and statement of our main theorems}
Birkhoff's ergodic theorem asserts that given a probability space equipped with a measure-preserving transformation $T\colon X\to X$ and a function $f\in L^1_{\mu}(X)$, the so-called conventional ergodic averages
\begin{equation}\label{Birconv}
\frac{1}{N}\sum_{n=1}^Nf(T^nx)\quad\text{converge as $N\to\infty$ for $\mu$-almost every $x\in X$.}
\end{equation}
This is a fundamental result in ergodic theory and both its implications as well as its emergent proof-techniques highlight a variety of connections of several mathematical fields. For example, it readily implies Kolmogorov's strong law of large numbers, and, in fact, it may well be the most efficient way of establishing it. Curiously, the modern and most robust proof of \eqref{Birconv} essentially reduces it to the weak-type $(1,1)$ bounds of the Hardy--Littlewood maximal function.
 
Establishing variants of the assertion \eqref{Birconv} for nonconventional ergodic averages taken along a sequence of natural numbers $(a_n)_{n\in\mathbb{N}}$, that is, proving $\mu$-a.e. convergence for $\frac{1}{N}\sum_{n=1}^Nf\circ T^{a_n}$ for functions belonging to some $L^p$-space, has been intensely studied in the last four decades. The first results in this direction appeared in the late 1980s in a series of breakthrough works \cite{B1,B2,Bourgain} by Bourgain where pointwise convergence for ergodic averages taken along polynomial and prime orbits was established. Therein a variety of different techniques from analytic number theory, harmonic analysis and probability theory was employed, and the methods introduced laid the foundations for furthering our understanding of pointwise convergence phenomena in ergodic theory. 

Before proceeding with our brief historical remarks, we introduce the following definition, which emerged from this line of investigations.
\begin{definition}[Universally $L^p$-good sequences]
\label{def:1}
We say that an integer sequence $(a_n)_{n\in\mathbb{N}}$ is
\emph{universally $L^p$-good} for some $p\in[1,\infty]$ if for every probability space
$(X,\mathcal{B},\mu)$, every invertible measure-preserving
transformation $T:X\to X$, and every $f\in L_{\mu}^p(X)$, the averages
\[
\frac{1}{N}\sum_{n=1}^Nf(T^{a_n}x)\quad\text{converge as $N\to\infty$ for $\mu$-almost every
$x\in X$.
}
\]
\end{definition}
With this definition in mind the aforementioned Bourgain's results may be formulated both more precisely and concisely. He showed that for every $P\in\mathbb{R}[x]$ the sequence $(\lfloor P(n)\rfloor )_{n\in\mathbb{N}}$ is universally $L^p$-good for every $p>1$ and also that the sequence of prime numbers is universally $L^p$-good for $p>\frac{1+\sqrt{3}}{2}$. Wierdl \cite{Wierdl} extended the latter result to the entire range $p>1$ $\footnote{This was later shown to be sharp \cite{LaV1}.}$. Since then, a wide range of classes of integer sequences has been studied in the context of universal $L^p$-goodness, and for a comprehensive overview of the current state of the art, we refer the reader to subsection~1.1 in \cite{DMW}. 

To the best of the author's knowledge, all known results in the literature are concerned with sequences of the form $(\lfloor f(n)\rfloor)_{n\in\mathbb{N}}$ with $f\colon \mathbb{N}\to\mathbb{R}$ being eventually monotonic. Our main result does not fall into this category, but, before making precise formulations, let us remind the reader a standard notion from Diophantine approximation that will play a role in the sequel.
\begin{definition}[Irrationality exponent]\label{irexpdef}
For every real number $\alpha$ we define its \textit{irrationality exponent}, denoted by $\omega(\alpha)$, as the supremum of the following set
\begin{equation}\label{irexpset}
\Big\{\omega\in(0,\infty):\,\text{there exists infinitely many coprime }(p,q)\in\mathbb{Z}\times\mathbb{N}\text{ with }0<\big|x-p/q\big|<q^{-\omega}\Big\}\text{.}
\end{equation}
\end{definition}
It is not difficult to see that every rational number has irrationality exponent $1$, and by Dirichlet's theorem every irrational has irrationality exponent at least $2$. Moreover, a standard Borel-Cantelli argument shows that almost all real numbers, with respect to the Lebesgue measure, have irrationality exponent $2$. 

One of our main theorems reads as follows.
\begin{theorem}\label{realanalyticformulation}
Assume that $\alpha$ is a real number with irrationality exponent $2$, $P\in\mathbb{R}[x]$ is a non-constant polynomial, and $\psi\colon\mathbb{R}\to\mathbb{R}$ is a $1$-periodic real-analytic function. Then the sequence
\begin{equation*}
\big(\lfloor P(n)\psi(\alpha n)\rfloor\big)_{n\in\mathbb{N}}\text{ is universally $L^p$-good for every $p\in(1,\infty]$.}
\end{equation*} 
\end{theorem}
For the sake of completeness, let us briefly recall that a function on $\mathbb{R}$ is called $1$-periodic if it has period $1$, and a function $\psi\colon\mathbb{R}\to\mathbb{R}$ is called real-analytic if it is smooth and for every $x\in\mathbb{R}$ there exists $\varepsilon>0$ such that the Taylor expansion of $\psi$ at $x$ coincides with $\psi$ on $(x-\varepsilon,x+\varepsilon)$.  

By our preceding comments, we immediately see that Theorem~$\ref{realanalyticformulation}$ implies that for a.e. $\alpha\in\mathbb{R}$ we have that for every $P\in\mathbb{R}[x]$ non-constant and $\psi\colon\mathbb{R}\to\mathbb{R}$ $1$-periodic and real-analytic, the sequence $(\lfloor P(n)\psi(\alpha n)\rfloor)_{n\in\mathbb{N}}$ is universally $L^p$-good for every $p>1$.

Apart from an ``almost everywhere'' result, Theorem~$\ref{realanalyticformulation}$ establishes pointwise convergence for ergodic averages modeled along  a variety of concrete orbits. For example, a standard result from the theory of continued fractions asserts that the irrationality exponent of $e$ is $2$, so we immediately$\footnote{Our examples here necessitate an application of our theorem for $\alpha=e/2$. Using Definition~$\ref{irexpdef}$ it is easy to check that for every $\alpha\in\mathbb{R}$ we have that $\omega(\alpha/2)=\omega(\alpha)$, and thus $\omega(e/2)=2$. In fact, for every irrational $\alpha$ and $a,b,c,d\in\mathbb{Z}$ with $ad-bc\neq0$, we have that $\omega\big(\frac{a\alpha+b}{d\alpha+d}\big)=\omega(\alpha)$, that is, the irrationality exponent is invariant under any fractional linear transformation with integer coefficients.}$ obtain that the following sequences are universally $L^p$-good for every $p>1$:
\begin{equation}\label{examples}
\lfloor n\sin( e\pi n)\rfloor\text{,}\quad \big\lfloor (\pi n^3+n)e^{\cos(e\pi n)}\big\rfloor\text{,}\quad \Big\lfloor (en^{10}+\pi n^2-10)e^{e^{\sin^2(e\pi n)}}\Big\rfloor\text{.}
\end{equation} 
In the same spirit, a famous result due to Roth \cite{Roth} from 1955 asserts that every real irrational algebraic number has irrationality exponent $2$, and thus Theorem~$\ref{realanalyticformulation}$ also addresses the following orbits
\begin{equation}\label{examples2}
\lfloor n\sin( \pi\sqrt{2} n)\rfloor\text{,}\quad\lfloor (\log 2)n^{5}e^{\sin(\pi (\sqrt[5]{7}+1) n)}\rfloor\text{,}\quad \Big\lfloor ( n^4+\pi^2n)\big(e^{e^{\cos(\pi \sqrt[3]{5} n)}}+100\sin^5(\pi \sqrt[3]{5} n)\big)\Big\rfloor\text{.}
\end{equation} 
Moreover, provided the verification of the recent breakthrough result asserting that the irrationality exponent of $\pi$ is $2$, see \cite{pirdm}, and using the fact that $\omega(\pi)=\omega\big(\frac{1}{2\pi}\big)$, we may additionally get that $\big(\lfloor n^d\sin(n)\rfloor\big)_{n\in\mathbb{N}}$ is universally $L^p$-good for every $p>1$ and $d\in\mathbb{N}$.

Finally, let us mention that Theorem~$\ref{realanalyticformulation}$ covers any orbit of the form
\[
\Big(\Big\lfloor P(n)\Big(\sum_{k=0}^K\big(a_k\cos^k(2\pi \alpha n)+b_k\sin^k(2\pi \alpha n)\big)\Big)\Big\rfloor\Big)_{n\in\mathbb{N}}\text{,} 
\]
with $P\in\mathbb{R}[x]$ a non-constant polynomial,  $K\in\mathbb{N}_0$, $(a_0,\dotsc,a_K),(b_0,\dotsc,b_K)\in\mathbb{R}^{K+1}$, and $\alpha$ a real number with irrationality exponent $2$.

We derive Theorem~$\ref{realanalyticformulation}$ as a corollary of the following more general theorem, covering orbits with $\psi$ belonging to a slightly wider class of $1$-periodic functions, and, more importantly, even allowing $\alpha$ to be much better approximated by rationals compared to typical numbers.
\begin{theorem}[Pointwise ergodic theorem along oscillating sequences]\label{Main}
Assume that $\psi\colon \mathbb{R}\to\mathbb{R}$ is a smooth $1$-periodic function such that none of its derivatives has a root of infinite order. Assume that $\alpha$ is a real number with irrationality exponent $\omega$, and $P\in\mathbb{R}[x]$ is a polynomial of degree $d\ge 1$ such that
$\omega<1+\sqrt{d+1}$. Then the sequence
\begin{equation*}
\big(\big\lfloor P(n)\psi(\alpha n)\rfloor\big)_{n\in\mathbb{N}}\text{ is universally $L^p$-good for every $p\in(1,\infty]$.}
\end{equation*}
\end{theorem}
Some remarks are in order.
\begin{enumerate}
\item[(i)] If $\alpha$ is a typical number in the sense of its Diophantine approximation, namely, such that $\omega(\alpha)=2$, then the condition  $\omega(\alpha)<1+\sqrt{\deg(P)+1}$ is satisfied for every non-constant polynomial $P$.
\item[(ii)] For rational $\alpha$ the condition again is trivially satisfied since $\omega(\alpha)=1$. In such a case, we note that the pointwise convergence of the ergodic averages modeled along  $\big(\big\lfloor P(n)\psi(\alpha n)\rfloor\big)_{n\in\mathbb{N}}$ can be easily deduced by passing to arithmetic progressions, exploiting the periodicity of $\psi$, and appealing to Bourgain's pointwise result for integer parts of polynomial orbits \cite{Bourgain}.
\item[(iii)] Theorem~$\ref{Main}$ yields examples of pointwise convergence for sequences of the form $\big(\big\lfloor P(n)\psi(\alpha n)\rfloor\big)_{n\in\mathbb{N}}$ for any irrational $\alpha$ with finite irrationality exponent, that is, for any non-Liouville irrational, provided that the polynomial is of sufficiently large degree, and more precisely, such that $\deg(P)>(\omega(\alpha)-1)^2-1$. For instance, even without the recent result from \cite{pirdm}, we may use the previous state-of-the-art upper bounds for the irrationality exponent of $\pi$, implying, in particular, that $\omega(\pi)<7.104$, see \cite{ZZ} for the more precise upper bound claimed therein, to obtain using the theorem above that $\big(\lfloor n^d\sin(n)\rfloor\big)_{n\in\mathbb{N}}$ is universally $L^p$-good for every $p>1$ and $d\ge 37$, since $1+\sqrt{d+1}\ge1+\sqrt{38}\ge 7.16>\omega(\pi)$. In the same spirit, using the fact that $\omega(\log 2)\le 3.575\le 1+\sqrt{6+1}$, see  \cite{marco}, our theorem covers, for example, $\big(\lfloor n^{d}\sin(n\pi \log 2)\rfloor\big)_{n\in\mathbb{N}}$ for every $d\ge 6$.
Analogous concrete orbits may be obtained whenever an upper bound for the irrationality exponent of $\alpha$ is available.\end{enumerate}

For the sake of completeness, let us give a short proof of the fact that this theorem readily implies Theorem~$\ref{realanalyticformulation}$.
\begin{proof}[Proof of Theorem~$\ref{realanalyticformulation}$ assuming Theorem~$\ref{Main}$]
If $\psi$ is such that none of its derivatives has a root of infinite order, we may directly apply Theorem~$\ref{Main}$ to deduce the $L^p$-goodness of $\big(\big\lfloor P(n)\psi(\alpha n)\rfloor\big)_{n\in\mathbb{N}} $. In the complementary case, namely, when $\psi$ is such that one of its derivatives does have a root of infinite order, by real analyticity, we get that this derivative must be  identically equal to $0$, and thus $\psi$ must be a polynomial. In turn, by periodicity, we get that $\psi$ is constant. If that constant is $0$, the result trivially holds, and, if not, the  $L^p$-goodness of the sequence follows from Bourgain's result, see \cite{Bourgain}.
\end{proof}

All our main theorems can be immediately derived from the following quantitative pointwise ergodic theorem.
\begin{theorem}[Quantitative pointwise ergodic theorem along oscillating sequences]\label{qMain}
Assume that $\psi\colon \mathbb{R}\to\mathbb{R}$ is a smooth $1$-periodic function such that none of its derivatives has a root of infinite order. Assume that $\alpha$ is an irrational number with irrationality exponent $\omega$, and $P\in\mathbb{R}[x]$ is a  polynomial of degree $d\ge 1$ such that
$\omega<1+\sqrt{d+1}$. Let $(X,\mathcal{B},\mu)$ be a $\sigma$-finite measure space equipped with an invertible measure-preserving transformation $T\colon X\to X$, and let for convenience
\begin{equation}\label{avoper}
\mathbf{A}_Nf(x)\coloneq \frac{1}{N}\sum_{n=1}^Nf\big(T^{\lfloor P(n)\psi(\alpha n)\rfloor}x\big)\text{.}
\end{equation}
Then for every $p\in(1,\infty)$ and every $f\in L^p_{\mu}(X)$ we have that
\begin{equation}\label{ptwisemine}
\mathbf{A}_Nf(x)\quad\text{converges as $N\to\infty$ for $\mu$-almost every $x\in X$.}
\end{equation}
Moreover,  for  every $p\in(1,\infty)$ and $r\in(2,\infty)$ there exists a constant $C=C(p,\alpha,\psi,P)$ such that 
\begin{equation}\label{vr}
\big\|V^r(\mathbf{A}_Nf:\,N\in\mathbb{N})\big\|_{L^p_{\mu}(X)}\le \frac{Cr}{r-2}\|f\|_{L^p_{\mu}(X)}\quad\text{for every $f\in L^p_{\mu}(X)$,}
\end{equation}
see \eqref{vrdef} for the definition of $r$-variations.
\end{theorem}
By standard arguments, it can be easily seen that the quantitative $r$-variational estimate in \eqref{vr} readily implies the qualitative assertion of $\mu$-a.e. convergence in \eqref{ptwisemine}. Additionally, it suffices to establish the estimate \eqref{vr} for the integer shift system, namely, for $(\mathbb{Z},\mathcal{P}(\mathbb{Z}),|\cdot|,x\to x-1)$, where $|\cdot|$ denotes the counting measure, since by the Calder\'on transference principle once such an estimate is available over $\mathbb{Z}$, one may immediately obtain it for arbitrary $\sigma$-finite measure-preserving systems. The averaging operators \eqref{avoper} considered on the integer shift system become the following convolution operators
\[
A_Nf(x)\coloneq \frac{1}{N}\sum_{n=1}^Nf\big(x-\lfloor P(n)\psi(\alpha n)\rfloor\big)\text{,}
\]
and their corresponding multipliers are the following normalized exponential sums
\begin{equation}\label{expsumshere}
\mathfrak{m}_N(\xi)\coloneq\frac{1}{N}\sum_{n=1}^Ne\big(\xi\lfloor P(n)\psi(\alpha n)\rfloor\big)\text{,}\quad\text{with $e(x)\coloneq e^{2\pi i x}$ for every $x\in\mathbb{R}$,}
\end{equation}
that is, we have that $\mathcal{F}_{\mathbb{Z}}[A_{N}f](\xi)
=\mathfrak{m}_N(\xi)\cdot\mathcal{F}_{\mathbb{Z}}[f](\xi)$ for every $\xi\in\mathbb{T}\coloneq \mathbb{R}/\mathbb{Z}$, see subsection~$\ref{notation}$ for our conventions for the Fourier transform.

Understanding the behavior of these exponential sums plays a key role in establishing the quantitative estimates of the form \eqref{vr}. Our exponential sums \eqref{expsumshere} have only the trivial obstruction, namely, the only frequency $\xi$ not making $|\mathfrak{m}_N(\xi)|$  small is the zero frequency $\xi\equiv0\Mod{1}$. In such a case, one employs the so-called one-frequency circle method, partitioning the torus $\mathbb{T}\equiv[-1/2,1/2)$ to one major and one minor arc, the former being an interval centered at $0$, the natural width of which here is $\simeq N^{-d+\delta}$ for some $\delta>0$, and the latter being its complement. 

Loosely speaking, for the one-frequency circle method in the context of pointwise convergence problems, the analysis of the exponential sums \eqref{expsumshere} consists of the following two tasks.
\begin{enumerate}
\item[1.]\textbf{Minor arc analysis:}
Establishing that $|\mathfrak{m}_N(\xi)|$ is appropriately uniformly decaying with respect to $N$ when $\xi$ belongs in the minor arc.
\item[2.]\textbf{Major arc analysis:} Approximating $\mathfrak{m}_N(\xi)$ for $\xi$ in the major arc by a multiplier which induces an operator more amenable to techniques yielding variational estimates (usually relying on the continuous theory).
\end{enumerate}
A key intermediate result capturing both the major and minor arc behavior of the exponential sums (see~ \eqref{majarcestcomplete} and \eqref{minarcestcompletedif}, respectively) is the following exponential sum estimate, which is, in some sense, the first of its kind. We elaborate in the fourth paragraph below.
\begin{proposition}[Exponential sum approximation]\label{fam}
Let $\alpha$ be an irrational with irrationality exponent $\omega$ and let $P\in\mathbb{R}[x]$ be a polynomial of degree $d\ge 1$ such that
$\omega<1+\sqrt{d+1}$. Assume that $\psi\colon \mathbb{R}\to\mathbb{R}$ is a smooth $1$-periodic function such that none of its derivatives has a root of infinite order. Then there exist positive constants $C=C(\alpha,P,\psi)$ and $\chi=\chi(\alpha,P,\psi)$ such that for every $N\in\mathbb{N}$ we have
\begin{equation}\label{fullestcomplete}
\sup_{\xi\in[-1/2,1/2)}\bigg|\frac{1}{N}\sum_{n=1}^Ne\big(\xi\lfloor P(n)\psi(\alpha n)\rfloor\big)-\int_0^1\bigg(\frac{1}{N}\int_0^Ne\big(\xi P(u)\psi(s)\big)du\bigg)ds\bigg|\le C N^{-\chi}\text{.}
\end{equation}
More precisely, there exists a positive constant $\delta=\delta(\omega)$ such that
\begin{equation}\label{majarcestcomplete}
\sup_{|\xi|\le N^{-d+\delta}}\bigg|\frac{1}{N}\sum_{n=1}^Ne\big(\xi\lfloor P(n)\psi(\alpha n)\rfloor\big)-\int_0^1\bigg(\frac{1}{N}\int_0^Ne\big(\xi P(u)\psi(s)\big)du\bigg)ds\bigg|\le C N^{-\chi}\text{,}
\end{equation}
\begin{equation}\label{minarcestcompletedif}
\sup_{N^{-d+\delta}<|\xi|\le1/2}\bigg|\frac{1}{N}\sum_{n=1}^Ne\big(\xi\lfloor P(n)\psi(\alpha n)\rfloor\big)\bigg|\le \frac{C}{2}N^{-\chi}
\end{equation}
and
\begin{equation}\label{minarcestcompleteeasy}
\sup_{N^{-d+\delta}<|\xi|\le1/2}\bigg|\int_0^1\bigg(\frac{1}{N}\int_0^Ne\big(\xi P(u)\psi(s)\big)du\bigg)ds\bigg|\le \frac{C}{2} N^{-\chi}\text{.}
\end{equation}
\end{proposition}
The estimate  \eqref{minarcestcompletedif} readily establishes that
\begin{equation}\label{easyL2}
\lim_{N\to\infty}\mathfrak{m}_N(\xi)= \left\{
\begin{array}{ll}     
0&
     \text{if }\xi\not\equiv 0\Mod{1}\\
      1
&\text{if }\xi\equiv 0\Mod{1}\text{,}\\
\end{array} 
\right.
\end{equation}
which, in turn, immediately implies that for every probability space $(X,\mathcal{B},\mu)$ equipped with an invertible measure-preserving transformation $T\colon X\to X$, the $L^2$-limit of $\mathbf{A}_Nf$ exists for every $f\in L^2_{\mu}(X)$, and coincides with the one of the conventional ergodic averages \eqref{Birconv}. The argument is standard and relies on Herglotz theorem followed by an application of the dominated convergence theorem using \eqref{easyL2}. In fact, the qualitative assertion \eqref{easyL2} which establishes the existence of the $L^2$-limit can be deduced from Theorem~4.1 and Theorem~4.2 in \cite{Trilogy1}  for a wider class of orbits than the one treated here. For instance, the order of differentiability of $\psi$ required is essentially $2d+1$, and one can allow any  non-Liouville irrational $\alpha$. In the context of $\mu$-a.e. convergence, estimates of the form \eqref{easyL2} are insufficient, and we require the much more quantitative and precise information about the behavior of the exponential sums Proposition~$\ref{fam}$ offers. 

Let us make some brief comments on the four estimates of the aforementioned proposition. Firstly, the last three estimates clearly imply the first. Secondly, each comes with its own hypotheses on $\alpha,P,\psi$, reflecting different sets of challenges on which we elaborate in the following strategy subsection. In Proposition~$\ref{fam}$ we have simply taken the intersection of all restrictions on $\alpha,P,\psi$ allowing all three estimates to be valid. 

Before discussing the strategy of our proof, we wish to comment on the major arc estimate \eqref{majarcestcomplete} since a novel phenomenon emerges. The exponential sums \eqref{expsumshere} have only the trivial obstruction, that is, \eqref{easyL2} holds. Usually, in such cases, the major arc analysis becomes straightforward, and the correct approximant of the normalized exponential sums is their direct continuous counterpart, i.e.:
\[
\frac{1}{N}\int_0^Ne(\xi P(u)\psi(\alpha u)\big)du\text{.}
\] 
This is not the case here. A decoupling effect takes place, and our approximant is 
\begin{equation}\label{ourapproximantnovel}
\int_0^1\bigg(\frac{1}{N}\int_0^Ne\big(\xi P(u)\psi(s)\big)du\bigg)ds\text{.}
\end{equation}
As alluded to earlier, approximations of the form \eqref{fullestcomplete} in the context of pointwise ergodic theorems allow one to deduce quantitative estimates of the form \eqref{vr} for the integer shift system from the analogous ones for the operators induced by the approximants. The utility of this reduction lies in the fact that the approximant multipliers are chosen so that their corresponding operators are more amenable to techniques yielding the aforementioned quantitative estimates, typically relying on the continuous theory.

The $\psi$-oscillation in our phase makes the standard major arc analysis through elementary comparisons of the exponential sums with their direct continuous counterpart (see, for example, page 6 and subsection~5.1 in \cite{DMW}) not possible here. Although not as explicit as in the aforementioned work, this type of elementary major arc approximations are implicit in seemingly every nonconventional pointwise ergodic theorem with corresponding exponential sums having only the trivial obstruction, see \cite{ncfull,mirek1,LDWT11,BY,frprimes}.

Nevertheless, our approximant multiplier in \eqref{ourapproximantnovel} is, indeed, amenable to techniques from the continuous theory, eventually yielding the desired variational estimates, through a different mechanism. The approximant multipliers \eqref{ourapproximantnovel} give rise to the following operators over $\mathbb{R}$
\begin{equation}\label{BNdefnew}
B_{N}f(x)\coloneq\int_{0}^1\Big(\frac{1}{N}\int_{0}^Nf\big(x-\psi(s)P(u)\big)du\Big)ds\text{,}
\end{equation}
which may be viewed as a superposition, or an average, of Radon-type operators.  Using the fact that the $r$-variations associated with averaging over a polynomial of degree $d$ is bounded on $L^p(\mathbb{R})$, $p\in(1,\infty]$ with a bound \textit{depending only on $p$ and $d$, but not on the coefficients of the polynomial}, one can  straightforwardly establish the variational estimates on $L^p(\mathbb{R})$ for the operators $(B_N)_{N\in\mathbb{N}}$. In fact, the argument to complete the aforementioned task is almost as simple as establishing the estimate $\|B_Nf\|_{L^p(\mathbb{R})}\le \|f\|_{L^p(\mathbb{R})}$, see the proof of Proposition~$\ref{contvarprop}$. These continuous estimates are a key input of the method, and together with the approximations in Proposition~$\ref{fam}$ allow us to ultimately conclude the proof of Theorem~$\ref{qMain}$.

Such coefficient-free $L^p$-estimates are established via a lifting technique, reminiscent of the argument yielding the Calder\'on transference principle. They were firstly developed in \cite{MT}, where higher dimensional analogues of Bourgain's pointwise result for polynomial orbits are established, and notably, are an important ingredient to the recent resolution of the multi-parameter variant
of the Bellow--Furstenberg problem \cite{BJMS}.  Although originally developed for quite different purposes, these estimates prove to be here precisely the necessary input for our variational estimates.    

\subsection{Strategy}\label{strategysubsection} The proof is naturally split into several tasks on which we briefly comment below.

\subsubsection{Major arc analysis}\label{majsubsubsec} The content of the major arc analysis, see Proposition~$\ref{Majarcprop}$, is establishing an estimate of the form \eqref{majarcestcomplete}. In this context, $\psi$ is allowed to be any $1$-periodic Lipschitz function. The proof is completed in several steps, and we give a quick and informal sketch here. Any passage described below should be understood as an appropriate estimate of the difference of the two expressions.
\begin{enumerate}
\item[(i)] Standard manipulations allow the floor removal here, and we make the first passage to a more amenable expression:
\[
\sum_{n\le N}e(\xi\lfloor P(n)\psi(\alpha n)\rfloor)\to\sum_{n\le N}e(\xi P(n)\psi(\alpha n))\eqcolon S_1\text{.}
\]
\item[(ii)] We perform an appropriate Diophantine approximation of $\alpha=p/q+\theta$, different for each $N$, and pass to arithmetic progressions
\[
S_1=\sum_{k=0}^{q-1}\sum_{\frac{1-k}{q}\le m\le\frac{N-k}{q}}e\big(\xi P(qm+k)\psi\big(\theta(qm+k)+kp/q\big)\big)\eqcolon S_2\text{,}
\]
where periodicity is exploited as follows:
\[
\psi(\alpha(qm+k))=\psi\big((\theta+p/q)(qm+k)\big)=\psi\big(\theta(qm+k)+kp/q\big)\text{.}
\] 
\item[(iii)] We will ensure that the Diophantine approximation is such that $N|\theta|\lesssim N^{-\tau}$ for some $\tau>0$ and thus using the fact that $|\theta(qm+k)|\lesssim |\theta| N\lesssim N^{-\tau}$ in the argument of the Lipschitz function $\psi$ above, we may remove the dependency on $m$ of the $\psi$-oscillation in the phase:
\[
S_2\to \sum_{k=0}^{q-1}\sum_{\frac{1-k}{q}\le m\le\frac{N-k}{q}}e\big(\xi P(qm+k)\psi(kp/q)\big)\eqcolon S_3
\]
\item[(iv)] The final two steps consist of approximating the two sums above by their continuous counterparts:
\[
S_3\to\frac{1}{q}\sum_{r=0}^{q-1}\int_{0}^Ne\big(\xi P(t)\psi(r/q)\big)dt\to\int_0^1\int_0^Ne\big(\xi P(u)\psi(s)\big)duds\text{.}
\]
\end{enumerate}
Arguably, the most crucial step is the second one, and it is inspired by the first step in the minor arc analysis (which, in turn, draws from ideas from \cite{Trilogy1}, see the following subsection). 

While the series of approximation arguments in the major arc analysis is reminiscent of the analogous one in the Hardy--Littlewood circle method for analyzing exponential sums with polynomial phases, their aim is entirely different. The only obstructing frequency for our exponential sums is the trivial one, and our maneuvers approximate $\alpha$ and not $\xi$ by rationals, serving the purpose of ``decoupling'' the oscillating part of the phase from the polynomial one. 
\subsubsection{Minor arc analysis}\label{minsubsubsec}
The minor arc analysis consists of establishing the estimates \eqref{minarcestcompletedif} and \eqref{minarcestcompleteeasy}. The latter is substantially easier and can be obtained essentially using van der Corput, see Proposition~$\ref{psireg}$. The most crucial step towards the former is establishing an estimate of the form
\begin{equation}\label{estminmain}
\Big|\sum_{N<n\le 2N}e(\xi P(n)\psi\big(\alpha n)\big)\Big|\lesssim N^{1-\chi}\text{,}\quad\text{for some $\chi>0$,}
\end{equation}
and for every $\xi,N$ with $N^{-d+\delta}\le|\xi|\le N$. After such an estimate is obtained, one may readily establish its non-dyadic variant and use it to obtain its ``floored'' counterpart, see Proposition~$\ref{minfull}$ and Proposition~$\ref{floorremovgeneric}$, respectively. Such arguments require some care, but are, nevertheless, standard and generic, in the sense that they have appeared in the literature and do not use the specific form of the phases, so let us focus the rest of the discussion on the estimate \eqref{estminmain}.

Our departure point is the beginning of the proof of Theorem~4.1 in \cite{Trilogy1}, see Case~1, page 172 therein, where oscillating sequences of similar type were considered in the context of equidistribution and discrepancy problems$\footnote{Although we only draw ideas from \cite{Trilogy1} it is worth mentioning that this work is the first part of a ``trilogy'' of papers \cite{Trilogy1,T2,T3} studying the distribution of wide classes of oscillating sequences modulo $1$.}$. We follow their first step in our exponential sum estimation, which is the same as the second step described in the major arc regime above, and consists of performing an appropriate Diophantine approximation of $\alpha=p/q+\theta$, different at each scale $N$, and pass to arithmetic progressions. We will again ensure that the error of the approximation is such that $|\theta|\lesssim N^{-(1+\tau)}$ for some $\tau>0$, and a passage to arithmetic progressions of common difference $q$ will allow us to estimate as follows
\begin{equation}\label{firstminslpit}
\begin{split}
\Big|\sum_{N<n\le 2N}e(\xi P(n)\psi\big(\alpha n)\big)\Big|\lesssim\sum_{r=0}^{q-1}\bigg|\sum_{\frac{N}{q}\le m\le \frac{2N}{q}}e\Big(\xi P(qm+r)\psi\big(\theta(qm+r)+rp/q\big)\Big)\bigg|+N^{-\chi}
\end{split}
\end{equation}
for some $\chi>0$. Although we aim to avoid making precise comments on the numerology of the various parameters in our proof in this informal discussion,  it is worth mentioning the following, hopefully providing the reader with a rough idea on how the notion of the irrationality exponent naturally emerges here: For every $N\in\mathbb{N}$, we perform a Diophantine approximation of $\alpha$ with $q=q(N)$ such that $q\le N^{\rho}$, for some carefully chosen $\rho\in(0,1)$. The upper bounds on the irrationality exponent of $\alpha$ immediately lead to lower bounds for $q$ and $|\theta|$, affording us more precise control of these quantities essential in the sequel.

The strategy after reaching \eqref{firstminslpit} departs from \cite{Trilogy1}, and instead of relying on refinements of van der Corput-type estimates we employ Vinogradov's method of estimating exponential sums. Specifically, we analyze the derivatives of the phase with respect to $m$, and letting
\[
F(x)=F(x;N,\xi,r)\coloneq \xi P(qx+r)\psi\big(\theta(qx+r)+rp/q\big)\text{,}
\]
we, loosely speaking, show that for ``most'' $r\in\{0,\dotsc,q-1\}$, and every $x\simeq N/q$ and $d\le j\lesssim1$ we have
\begin{equation}\label{derF}
N^{-\sigma}\cdot \big(|\xi|q^j |\theta|^{j-d}\big)\lesssim |F^{(j)}(x)|\lesssim |\xi|q^j |\theta|^{j-d}\text{,}\quad\text{for some $\sigma>0$.}
\end{equation}
Such an estimate is achieved by exploiting, among other things, the fact that no derivative of $\psi$ can have a root of infinite order, see Lemma~$\ref{elementarybutdifficult}$. 

Using the trivial estimate for atypical residue classes in \eqref{firstminslpit}, the task is reduced to estimating 
\[
\bigg|\sum_{\frac{N}{q}\le m\le \frac{2N}{q}}e\Big(\xi P(qm+r)\psi\big(\theta(qm+r)+rp/q\big)\Big)\bigg|
\]
for $r$'s such that \eqref{derF} holds, and for precisely this purpose we employ Vinogradov's method. Although $F$ does  not have derivatives decaying as if it were a polynomial, which is the setting for which Vinogradov's method is usually adapted to, the estimate \eqref{derF} suggests (ignoring the  error factor $N^{-\sigma}$) that after taking more than $d$ derivatives, the decay of the size for an extra derivation corresponds to an additional multiplicative factor $q|\theta|$ (instead of $(N/q)^{-1}$, which would correspond to polynomial behavior). Since $|\theta|\le N^{-1}$, this means that the derivatives decay faster than the polynomial heuristic. Nevertheless, using the framework of Vinogradov's method for estimating exponential sums in a very general form, see  Proposition~$\ref{genericvin}$, the state-of-the-art Vinogradov Mean Value Theorem bounds of
Wooley \cite{W} and Bourgain-Demeter-Guth \cite{BDG}, and the derivative control for one precisely chosen derivative depending on $N$ and $\xi$, we are able to conclude. Finally, let us mention that although discarding information about every derivative except for one$\footnote{Given $\xi,N$ we do use the control \eqref{derF} affords us for one precisely chosen $j_0=j_0(\xi,N)$ when using the Vinogradov method, see Step~5 in the proof of Proposition~$\ref{minarcprop}$. But our choice $j_0=j_0(\xi,N)$ takes multiple different values as $\xi,N$ range in the minor arc regime $N^{-d+\delta}\le|\xi|\le N$, so our argument ultimately utilizes information about a large number of derivatives.}$ goes against the spirit of Vinogradov's method, this is not the first successful employment of a maneuver of this kind, see for example \cite{hb}.
\subsubsection{Variational estimates}
Having the approximations of Proposition~$\ref{fam}$ in our disposal, the key additional input for establishing the quantitative estimate of Theorem~$\ref{qMain}$ is the following variational estimate on $\mathbb{R}$ for the continuous analogue of the approximant operator
\begin{equation}\label{rvarnewnew}
\big\|V^r(B_{N}f:\,N\in\mathbb{N})\big\|_{L^p(\mathbb{R})}\lesssim_{p,d}\frac{r}{r-2}\|f\|_{L^p(\mathbb{R})}\text{,}\quad\text{where the operators $B_N$ are defined \eqref{BNdefnew}.}
\end{equation}
We have already commented on the proof of this assertion as a direct corollary of the coefficient-free variational $L^p(\mathbb{R})$-bounds for averaging operators of Radon-type, see the two paragraphs preceding subsection~$\ref{strategysubsection}$.
 
After such an estimate is obtained, one may establish a suitable variant on $\ell^p(\mathbb{Z})$ using a standard sampling principle, see Section~3.3 in \cite{marvar}, and Section~2 in \cite{TransferenceLp}. Finally, the appropriate discrete counterpart of \eqref{rvarnewnew} together with our approximations in Proposition~$\ref{fam}$ ultimately allow us to establish \eqref{vr} for the integer shift system, which by the Calder\'on transference principle gives the variational estimate for the arbitrary $\sigma$-finite measure-preserving system. 
\\

We have organized the paper as follows. In Section~$\ref{majsec}$ and Section~$\ref{minsec}$ we perform the major and minor arc analysis, respectively, the latter of which concludes with a quick proof of our exponential sum approximation result, namely, Proposition~$\ref{fam}$. In Section~$\ref{contandconclude}$ we establish variational  estimates in the continuous setting for our approximant in the aforementioned proposition, see Proposition~$\ref{contvarprop}$, and in the final section we combine everything to conclude the proof of Theorem~$\ref{qMain}$, which as we saw readily implies all our main theorems.
\subsection{Notation}\label{notation}
For any $x\in\mathbb{R}$ we use the standard notation
\[
\lfloor x\rfloor=\max\{n\in\mathbb{Z}:\,n\le x\}\text{,}\quad\{x\}=x-\lfloor x\rfloor\text{,} \quad\|x\|=\min\{|x-n|:\,n\in\mathbb{Z}\}\text{.}
\]
Through nearly the entirety of the paper we identify $\mathbb{T}\coloneq\mathbb{R}/\mathbb{Z}$ with $[-1/2,1/2)$ so that for every $\xi\in\mathbb{T}\equiv[-1/2,1/2)$ we have $\|\xi\|=|\xi|$. For every $t\in[1,\infty)$ we let $[t]\coloneqq\{1,2,\dotsc,\lfloor t\rfloor\}$.  If $A,B$ are two non-negative quantities, we write $A\lesssim B$ or $B \gtrsim A$  to denote the existence of a positive constant $C$ such that $A\le C B$. When $A\lesssim B$ and $A \gtrsim B$ we write $A\simeq B$. Whenever we wish to highlight the dependence of the implicit constant $C$ on some parameters $\alpha_1,\dotsc,\alpha_{k}$, we write $A\lesssim_{\alpha_1,\dotsc,\alpha_{k}} B$. We let $\lambda$ and $\lambda_{\mathbb{T}}$ denote the Lebesgue measure on $\mathbb{R}$ and $\mathbb{T}$, respectively.

For any $f\in\ell^1(\mathbb{Z})$ we define its Fourier transform by
\[
\mathcal{F}_{\mathbb{Z}}f(\xi)\coloneq\sum_{n\in\mathbb{Z}}f(n)e(\xi n)\text{,}\quad\xi\in\mathbb{T}\text{,}
\] 
and for any $g\in L^1(\mathbb{T})$ we define its inverse Fourier transform by
\[
\mathcal{F}_{\mathbb{Z}}^{-1}g(x)\coloneq\int_{\mathbb{T}}g(\xi)e(-\xi x) d \xi\text{,}\quad x\in\mathbb{Z}\text{.} 
\]
For any $h\in L^1(\mathbb{R})$ we define the Fourier transform and the inverse Fourier transform of $h$ by
\[
\mathcal{F}_{\mathbb{R}}[h](\xi)\coloneq\int_{\mathbb{R}}f(x)e(\xi x)d x
\quad\text{and}\quad\mathcal{F}^{-1}_{\mathbb{R}}[h](x)\coloneq\int_{\mathbb{R}}f(\xi)e(-\xi x)d\xi\text{,}\quad\text{ respectively.}
\]
For any $\mathbb I\subseteq \mathbb R$, any family
$(a_t)_{t\in\mathbb I}=(a_t: t\in\mathbb I)\subseteq \mathbb C$,
and any exponent $r \in [1, \infty)$, we define the $r$-variation seminorm of the family by
\begin{align}
\label{vrdef}
V^{r}(a_t: t\in\mathbb I) \coloneqq 
\sup_{J\in\mathbb{N}} \sup_{\substack{t_{0}<\dotsb<t_{J}\\ t_{j}\in\mathbb I}}
\bigg(\sum_{j=1}^J  |a_{t_{j}}-a_{t_{j-1}}|^{r} \bigg)^{1/r},
\end{align}
where the supremum is taken over all finite increasing sequences in
$\mathbb I$, and is set by convention to equal zero if
$|\mathbb I| \leq 1$.

\section*{Acknowledgments}
The author would like to thank Andreas Koutsogiannis and Konstantinos Tsinas for suggesting the problem, Borys Kuca for insightful conversations which led to several improvements of the exposition, as well as B{\l}a{\.z}ej Wr{\'o}bel and Nikos Frantzikinakis for their constant support and encouragement. I would also like to thank Mariusz Mirek for his guidance and numerous invaluable discussions throughout the years on related problems.
\section{Major arc estimates}\label{majsec}
In this section we perform the major arc analysis, namely, we establish an estimate of the form \ref{majarcestcomplete}. More precisely, we prove the following.
\begin{proposition}[Major arc estimates]\label{Majarcprop}
Let $\alpha\in\mathbb{R}$ be such that $\omega(\alpha)\eqcolon\omega\in[2,\infty)$, let $P\in\mathbb{R}[x]$ be  of degree $d\ge 1$, and let $\psi\colon\mathbb{R}\to\mathbb{R}$ be a $1$-periodic Lipschitz function. Then for every $\delta\in\big(0,\frac{1}{2\omega-1}\big)$ there exist positive constants $C=C(\alpha,\delta,P)$ and $\chi=\chi(\delta,\omega)$ such that for every $N\in\mathbb{N}$ and $\xi\in\mathbb{R}$ with $|\xi|\le N^{-d+\delta}$ we have
\begin{equation}\label{superpositionmajarc}
\bigg|\frac{1}{N}\sum_{n=1}^Ne\big(\xi \lfloor P(n)\psi(\alpha n)\rfloor\big)-\int_0^1\bigg(\frac{1}{N}\int_0^Ne\big(\xi P(t)\psi( s)\big)dt\bigg)ds\bigg|\le C \big(\|\psi\|_{\Lip}+1\big)N^{-\chi}\text{,}
\end{equation}
where 
\[
\|\psi\|_{\Lip}\coloneq \sup_{x\in\mathbb{R}}|\psi(x)|+\sup_{\substack{x\neq y\\x,y\in\mathbb{R}}}\frac{|\psi(y)-\psi(x)|}{|x-y|}\text{.} 
\] 
\end{proposition}
Before providing the proof, let us collect a convenient reformulation of the notion of the irrationality exponent used in the sequel.
\begin{lemma}\label{expbreath}
For every irrational number $\alpha$ such that $\omega(\alpha)<\infty$ and every $\varepsilon>0$ there exists a constant $c(\alpha,\varepsilon)\in(0,1)$ such that
\begin{equation}\label{statement}
\big|\alpha-p/q\big|\ge \frac{c}{q^{\omega(\alpha)+\varepsilon}}\quad\text{for every $(p,q)\in\mathbb{Z}\times\mathbb{N}$ with $\gcd(p,q)=1$.}
\end{equation}
\end{lemma}
\begin{proof}
The claim is immediate from the definition, but we, nevertheless, include a short proof for completeness. Since $\omega(\alpha)$ is the supremum of the set in \eqref{irexpset} we get that
\[
|S(\alpha,\varepsilon)|\coloneq\big|\big\{(p,q)\in\mathbb{Z}\times\mathbb{N}:\,\gcd(p,q)=1\text{ and }0<\big|\alpha-p/q\big|<q^{-\omega(\alpha)-\varepsilon}\big\}\big|<\infty\text{,}
\]
and taking $c(\alpha,\varepsilon)\coloneq \min\big(\big\{q^{\omega(\alpha)+\varepsilon}\big|\alpha-p/q\big|:\,(p,q)\in S(\alpha,\varepsilon)\big\}\cup\{1/2\}\big)\in(0,1)$ completes the proof.
\end{proof}
We are now ready to give the proof of Proposition~$\ref{Majarcprop}$.
\begin{proof}[Proof of Proposition~$\ref{Majarcprop}$]
Let $\alpha,\omega,P,d,\psi$ and $\delta$ be as in the proposition's statement and fix
\[
\chi=\chi(\delta,\omega)\coloneq\frac{1}{2}\bigg(\frac{1}{2\omega-1}-\delta\bigg)>0\text{.}
\]
We allow all implicit constants to depend on $\alpha,\delta,P$. Let $N\in\mathbb{N}$ and $\xi\in\mathbb{R}$ be such that $|\xi|\le N^{-d+\delta}$. 
\smallskip
\paragraph{\textbf{Step 1: Removing the floor.}}
 The following manipulations are standard in the major arc regime:
\begin{equation}
\begin{split}
\label{removefloor}
\Big|\sum_{n\le N}e(\xi\lfloor P(n)\psi(\alpha n)\rfloor)-\sum_{n\le N}e(\xi P(n)\psi(\alpha n))\Big|\le\sum_{n\le N}|e(\xi\lfloor P(n)\psi(\alpha n)\rfloor)-e(\xi P(n)\psi(\alpha n))|&
\\
\lesssim  |\xi|\sum_{n\le N}\big|P(n)\psi(\alpha n)-\lfloor P(n)\psi(\alpha n)\rfloor\big|\le|\xi| N\le  N^{-d+\delta+1}\le N^{-\frac{1}{2\omega-1}+\delta+1}\le N^{1-\chi}&\text{,} 
\end{split}
\end{equation}
since $d\ge 1\ge(2\omega-1)^{-1}$, because $\omega\ge 2$. This constitutes the first approximant, and for the sake of book-keeping we let
\begin{equation}\label{E1er}
S_0\coloneq\sum_{n\le N}e(\xi\lfloor P(n)\psi(\alpha n)\rfloor)\text{, } S_1\coloneq \sum_{n\le N}e(\xi P(n)\psi(\alpha n))\text{ and }E_1\coloneq S_0-S_1\text{, so that }|E_1|\lesssim N^{1-\chi}\text{.}
\end{equation}
\smallskip
\paragraph{\textbf{Step 2: Diophantine approximation for $\alpha$.}} We perform a careful Diophantine approximation for $\alpha$. Since $\chi<\frac{1}{2\omega-1}-\delta$, there exists $\omega'=\omega'(\omega,\delta)>\omega\ge 2$ such that
\begin{equation}\label{chichoicedelta}
\chi<\frac{1}{2\omega'-1}-\delta\text{.}
\end{equation}
We fix such an $\omega'$ and also fix 
\begin{equation}\label{rhores}
\rho=\rho(\omega,\delta)\coloneq1-\frac{1}{2\omega'-1}\text{,}\quad\text{and note that by \eqref{chichoicedelta} we have $\rho+\delta<1-\chi$.}
\end{equation}
For every $N\in\mathbb{N}$ we apply the Dirichlet's principle to deduce the existence of a pair $(p,q)=(p(N),q(N))\in\mathbb{Z}\times\mathbb{N}$ with $\gcd(p,q)=1$, $q\le N^{\rho}$ and such that
\begin{equation}\label{thetaupbeg}
\Big|\alpha-\frac{p}{q}\Big|\le \frac{1}{qN^{\rho}}\text{.}
\end{equation}
Since $\alpha$ has irrationality exponent $\omega$ and $\omega'>\omega$,  Lemma~$\ref{expbreath}$ guarantees the existence of a constant $c\in(0,1)$, depending only on $\alpha,\delta$, such that
\[
\Big|\alpha-\frac{p}{q}\Big|\ge \frac{c}{q^{\omega'}}\text{.}
\]
This immediately yields that 
\begin{equation}\label{qlower}q\gtrsim N^{\frac{\rho}{\omega'-1}}=N^{2(1-\rho)}\text{,}
\end{equation}
since
\begin{equation}\label{qlowerboundnew}
\frac{1}{qN^{\rho}}\ge\frac{c}{q^{\omega'}}\iff q^{\omega'-1}\ge cN^{\rho}\Longrightarrow q\gtrsim N^{\frac{\rho}{\omega'-1}}=N^{2(1-\rho)}\text{,}
\end{equation}
where the last equality is justified by \eqref{rhores} as follows
\begin{equation}\label{easycalcrho}
2(1-\rho)=\frac{\rho}{\omega'-1}\iff \frac{2}{2\omega'-1}=\frac{1-\frac{1}{2\omega'-1}}{\omega'-1} 
\iff 2=\frac{2\omega'-2}{\omega'-1}\text{.}
\end{equation}
Similarly, letting $\theta=\theta(N)\coloneq \alpha-\frac{p}{q}$ be the error of the Diophantine approximation, we see that 
\begin{equation}\label{thetaupnew}
N|\theta|\le\frac{N}{qN^{\rho}}=q^{-1}N^{1-\rho}\lesssim N^{-2(1-\rho)} N^{1-\rho}=N^{-(1-\rho)}\text{.}
\end{equation}
We now split the summation into arithmetic progressions
 \begin{equation}
\begin{split}
S_1=&\sum_{n=1}^{N}e(\xi P(n)\psi(\alpha n))
=\sum_{k=0}^{q-1}\sum_{\frac{1-k}{q}\le m\le \frac{N-k}{q}}e\big(\xi P(qm+k)\psi(\alpha (qm+k))\big)
\\
=&\sum_{k=0}^{q-1}\sum_{\frac{1-k}{q}\le m\le\frac{N-k}{q}}e\big(\xi P(qm+k)\psi\big(\theta(qm+k)+kp/q\big)\big)\text{,}
\end{split}
\end{equation}
where for the last equality we used the periodicity of $\psi$ as follows
\begin{equation}\label{psiperiod}
\psi(\alpha(qm+k))=\psi\big((\theta+p/q)(qm+k)\big)=\psi\big(\theta(qm+k)+pm+kp/q\big)=\psi\big(\theta(qm+k)+kp/q\big)\text{.}
\end{equation}
\smallskip
\paragraph{\textbf{Step 3: Making the $\psi$-oscillation independent of $m$.}} We exploit the fact that $N|\theta|\lesssim N^{-(1-\rho)}$, see \eqref{thetaupnew}, in order to replace $\psi\big(\theta(qm+k)+kp/q\big)$ in the phase, with $\psi(kp/q)$. More precisely, we let
\begin{equation*}
\begin{split}
S_1=&
\sum_{k=0}^{q-1}\sum_{\frac{1-k}{q}\le m\le \frac{N-k}{q}}e\big(\xi P(qm+k)\psi(kp/q)\big)
\\
+&\sum_{k=0}^{q-1}\sum_{\frac{1-k}{q}\le m\le \frac{N-k}{q}}\Big(e\big(\xi P(qm+k)\psi(\theta (qm+k)+kp/q)\big)-e\big(\xi P(qm+k)\psi(kp/q)\big)\Big)\eqcolon S_2+E_2\text{,}
\end{split}
\end{equation*}
and bound the new error term as follows
\begin{equation}\label{E2er}
\begin{split}
|E_2|\le&\sum_{k=0}^{q-1}\sum_{\frac{1-k}{q}\le m\le \frac{N-k}{q}}\Big|e\big(\xi P(qm+k)\psi(\theta (qm+k)+kp/q)\big)-e\big(\xi P(qm+k)\psi(kp/q)\big)\Big|
\\
\lesssim& \sum_{k=0}^{q-1}\sum_{\frac{1-k}{q}\le m\le \frac{N-k}{q}}\big|\xi P(qm+k)\psi(\theta (qm+k)+kp/q)-\xi P(qm+k)\psi(kp/q)\big|
\\
= &|\xi| \sum_{k=0}^{q-1}\sum_{\frac{1-k}{q}\le m\le \frac{N-k}{q}}|P(qm+k)|\cdot\big| \psi(\theta (qm+k)+kp/q)-\psi(kp/q)\big|
\\
\lesssim &|\xi| \sum_{k=0}^{q-1}\sum_{\frac{1-k}{q}\le m\le \frac{N-k}{q}}(qm+k)^d\|\psi\|_{\Lip}| \theta (qm+k)|
\\
\lesssim& \|\psi\|_{\Lip}|\xi| \sum_{k=0}^{q-1}\sum_{\frac{1-k}{q}\le m\le \frac{N-k}{q}}N^d| \theta |N
\\
\lesssim& \|\psi\|_{\Lip}|\xi| N^{d+1}N^{-(1-\rho)}\lesssim \|\psi\|_{\Lip} N^{\delta+\rho}\lesssim\|\psi\|_{\Lip} N^{1-\chi}\text{,}
\end{split}
\end{equation}
where we have used the fact that $|P(x)|\lesssim x^{d}$, as well as \eqref{thetaupnew} and \eqref{rhores}. Thus, our main term now becomes $S_2$, which is amenable to standard approximation arguments for its two sums  by their continuous counterparts. We perform these approximations in the final two steps.
\smallskip
\paragraph{\textbf{Step 4: Passing to the continuous counterpart in $m$.}} Note that
\begin{equation}
\begin{split}
S_2=& \sum_{k=0}^{q-1}\sum_{\frac{1-k}{q}\le m\le \frac{N-k}{q}}e\big(\xi P(qm+k)\psi(kp/q)\big)=\sum_{k=0}^{q-1}\int_{-\frac{k}{q}}^{\frac{N-k}{q}}e\big(\xi P(qy+k)\psi(kp/q)\big)dy
\\
+&\sum_{k=0}^{q-1}\Big(\sum_{\frac{1-k}{q}\le m\le\frac{N-k}{q}}e\big(\xi P(qm+k)\psi(kp/q)\big)-\int_{-\frac{k}{q}}^{\frac{N-k}{q}}e\big(\xi P(qy+k)\psi(kp/q)\big)dy\Big)\eqcolon S_3+E_3\text{.}
\end{split}
\end{equation}
It is not difficult to see that
\begin{equation}
\begin{split}
S_3=\frac{1}{q}\sum_{k=0}^{q-1}\int_{0}^Ne\big(\xi P(t)\psi(kp/q)\big)dt=\frac{1}{q}\sum_{r=0}^{q-1}\int_{0}^Ne\big(\xi P(t)\psi(r/q)\big)dt\text{.}
\end{split}
\end{equation}
The first equality is a change of variables, and the second can be justified using the fact that the map 
\[
\mathbb{Z}\ni r\to\psi(r/q)\in\mathbb{R}\quad\text{depends only on the value of $r\Mod{q}$ by the periodicity of $\psi$,}
\]
as well as the fact that the map 
\[
\mathbb{Z}/(q\mathbb{Z})\ni k\to pk\in\mathbb{Z}/(q\mathbb{Z})\quad\text{is a bijection, since $\gcd(p,q)=1$. }
\]
For the error term we note that
\begin{equation}
\begin{split}
|E_3|\le&\sum_{k=0}^{q-1}\Big|\sum_{\frac{1-k}{q}\le m\le\frac{N-k}{q}}e\big(\xi P(qm+k)\psi(kp/q)\big)-\int_{-\frac{k}{q}}^{\frac{N-k}{q}}e\big(\xi P(qy+k)\psi(kp/q)\big)dy\Big|
\\
\le&\sum_{k=0}^{q-1}\Big|\sum_{ m\le \frac{N}{q}}\int_{m-1}^{m}\Big(e\big(\xi P(qm+k)\psi(kp/q)\big)-e\big(\xi P(qy+k)\psi(kp/q)\big)\Big)dy\Big|
\\
+&\sum_{k=0}^{q-1}\Big(\sum_{\frac{1-k}{q}\le m< 1}1+\sum_{\frac{N-k}{q}<m\le\frac{N}{q} }1\Big)+\sum_{k=0}^{q-1}\Big(\Big|\int_{\lfloor N/q\rfloor}^{N/q}e\big(\xi P(qy+k)\psi(kp/q)\big)dy\Big|+2\Big)
\\
\lesssim&q+\sum_{k=0}^{q-1}\sum_{m\le \frac{N}{q}}\int_{m-1}^{m}\big|e\big(\xi P(qm+k)\psi(kp/q)\big)-e\big(\xi P(qy+k)\psi(kp/q)\big)\big|dy
\\
\lesssim&q+\sum_{k=0}^{q-1}\sum_{m\le \frac{N}{q}}\int_{m-1}^{m}\big|\xi P(qm+k)\psi(kp/q)-\xi P(qy+k)\psi(kp/q)\big|dy
\\
\lesssim&q+|\xi| \sum_{k=0}^{q-1}\big|\psi(kp/q)\big|\sum_{m\le \frac{N}{q}}\int_{m-1}^{m}|P(qm+k)-P(qy+k)|dy
\\
\lesssim& q+|\xi| \|\psi\|_{\Lip}\sum_{k=0}^{q-1}\sum_{m\le N/q}\int_{m-1}^m|P(qm+k)-P(qy+k)|dy\text{.}
\end{split}
\end{equation}
To conclude, we note that by the Mean Value Theorem, for every $k\in\{0,\dotsc,q-1\}$, $m\in[N/q]$ and $y\in(m-1,m)$, there exists $y_{m,k}\in(y,m)$ such that
\[
|P(qm+k)-P(qy+k)|=q|P'(qy_{m,k}+k)(m-y)| \lesssim q(N+q)^{d-1}\text{,}
\]
and thus
\begin{equation}\label{E3er}
\begin{split}
|E_3|\lesssim& q+|\xi| \|\psi\|_{\Lip}\sum_{k=0}^{q-1}\sum_{m\le N/q}\int_{m-1}^mq(N+q)^{d-1}dy
\le q+|\xi|\|\psi\|_{\Lip}q(N/q)q(2N)^{d-1}
\\
\lesssim& q+\|\psi\|_{\Lip}|\xi|N^dq \lesssim \big(\|\psi\|_{\Lip}+1\big)N^{\delta}q
\lesssim \big(\|\psi\|_{\Lip}+1\big)N^{\delta+\rho}\le \big(\|\psi\|_{\Lip}+1\big)N^{1-\chi}\text{.}
\end{split}
\end{equation}
Thus we focus on $S_3$.
\smallskip
\paragraph{\textbf{Step 5: Passing to the continuous counterpart in $k$.}} For the final step, we note that
\begin{equation}
\begin{split}
S_3=&\int_0^N\frac{1}{q}\sum_{r=0}^{q-1}e\big(\xi P(t)\psi(r/q)\big)dt=\int_{0}^N\frac{1}{q}\int_{0}^{q}e\big(\xi P(t)\psi(u/q)\big)dudt
\\
+&\frac{1}{q}\int_0^{N}\bigg(
\sum_{r=0}^{q-1}e\big(\xi P(t)\psi(r/q)\big)-\int_{0}^{q}e\big(\xi P(t)\psi(u/q)\big)du\bigg)dt\eqcolon S_4+E_4\text{.}
\end{split}
\end{equation}
After a change of variables followed by a change of the order of integration we get that
\[
S_4=\int_{0}^{1}\int_{0}^Ne\big(\xi P(t)\psi(s)\big)dtds\text{,}
\]
and the final error term can be handled using \eqref{qlower} as follows
\begin{equation}\label{E4er}
\begin{split}
|E_4|\le&
 \frac{1}{q}\int_0^{N}\bigg|
\sum_{r=0}^{q-1}\int_r^{r+1}\Big(e\big(\xi P(t)\psi(r/q)\big)-e\big(\xi P(t)\psi(u/q)\big)\Big)du\bigg|dt
\\
\lesssim&  \frac{1}{q}\int_0^{N}
\sum_{r=0}^{q-1}\int_r^{r+1}\big|\xi P(t)\psi(r/q)-\xi P(t)\psi(u/q)\big|dudt
\\
\lesssim& \frac{|\xi| N^d}{q}\int_0^{N}
\sum_{r=0}^{q-1}\int_r^{r+1}|\psi(r/q)-\psi(u/q)|dudt
\\
\lesssim& \frac{\|\psi\|_{\Lip} N^{\delta}}{q}\int_0^{N}
\sum_{r=0}^{q-1}\int_r^{r+1}q^{-1}dudt\lesssim \|\psi\|_{\Lip} N^{\delta+1}q^{-1}
\\
\lesssim&\|\psi\|_{\Lip}N^{1+\delta-2+2\rho}\lesssim \|\psi\|_{\Lip}N^{\delta+\rho-(1-\rho)}\lesssim\|\psi\|_{\Lip}N^{1-\chi}
\text{.}
\end{split}
\end{equation}
We can immediately conclude now since
\begin{equation}
\begin{split}
\Big|\sum_{n=1}^Ne&(\xi \lfloor P(n)\psi(\alpha n)\rfloor)-\int_0^1\int_0^Ne(\xi P(t)\psi( s))dtds\Big|
=|S_0-S_4|
\\
\le& \sum_{i\in[4]}|S_{i-1}-S_i|
=\sum_{i\in[4]}|E_i|\lesssim\big(\|\psi\|_{\Lip}+1\big)N^{1-\chi}\text{,}\quad\text{by \eqref{E1er}, \eqref{E2er}, \eqref{E3er}, \eqref{E4er}.
}
\end{split}
\end{equation}
\end{proof}
\section{Minor arc estimates}\label{minsec}
This section is devoted to the minor arc analysis, namely to establishing estimates of the form \eqref{minarcestcompletedif} and \eqref{minarcestcompleteeasy}. While the latter can be obtained relatively straightforwardly, establishing the former constitutes the most intricate part of the present work. The main result of this section is the following.
\begin{proposition}[Minor arc estimates]
\label{minarcpropfull}
Assume that $\psi\colon \mathbb{R}\to\mathbb{R}$ is a smooth $1$-periodic function such that none of its derivatives has a root of infinite order. Let $\alpha\in\mathbb{R}$ be such that  $\omega(\alpha)\eqcolon \omega\in[2,\infty)$, let $P\in\mathbb{R}[x]$ be of degree $d\ge 1$ and assume that
\begin{equation}\label{mudrestrmain}
\omega<1+\sqrt{d+1}\text{.}
\end{equation}
Then for every $\delta>0$ there exist positive constants $\chi=\chi(\alpha,P,\psi,\delta)$ and $C=C(\alpha,P,\psi,\delta)$ such that for every $N\in\mathbb{N}$ and $\xi\in\mathbb{R}$ with $N^{-d+\delta}\le |\xi|\le 1/2$ we have
\begin{equation}\label{minarxestfloored}
\Big|\sum_{n=1}^Ne\big(\xi \lfloor P(n)\psi(\alpha n)\rfloor\big)\Big|\le CN^{1-\chi}\text{.}
\end{equation}
\end{proposition}
The most important intermediate step towards the previous proposition is the following ``unfloored'' dyadic variant of the exponential sum estimate.
\begin{proposition}\label{minarcprop}
Assume that $\psi\colon \mathbb{R}\to\mathbb{R}$ is a smooth $1$-periodic function such that none of its derivatives has a root of infinite order. Let $\alpha\in\mathbb{R}$ be such that  $\omega(\alpha)\eqcolon \omega\in[2,\infty)$, let $P\in\mathbb{R}[x]$ be of degree $d\ge 1$ and assume that
\begin{equation}\label{mudrestr}
\omega<1+\sqrt{d+1}\text{.}
\end{equation}
Then for every $\delta>0$ there exist positive constants $\chi=\chi(\alpha,P,\psi,\delta)$ and $C=C(\alpha,P,\psi,\delta)$ such that for every $N\in[1,\infty)$ and $\xi\in\mathbb{R}$ with $N^{-d+\delta}\le |\xi|\le N$ we have
\begin{equation}\label{goal1expsummin}
\Big|\sum_{N<n\le 2N}e(\xi P(n)\psi\big(\alpha n)\big)\Big|\le CN^{1-\chi}\text{.}
\end{equation}
\end{proposition}
\subsection{Key proof ingredients} In this subsection we collect the following three results before embarking on the proof of Proposition~$\ref{minarcprop}$. The first is the Vinogradov Mean Value Theorem~\cite{W,BDG}, the second is a technical intermediate step of Vinogradov's method for estimating exponential sums, and the last one offers the precise way we translate the fact that $\psi$ is such that none of its derivatives has a root of infinite order.
\begin{theorem}[Vinogradov Mean Value Theorem]\label{VMVTf}
For every $s,k\in\mathbb{N}$ and $\varepsilon>0$, there exists a positive constant $C=C(s,k,\varepsilon)$ such that for every $X\in\mathbb{N}$ we have
\begin{equation}
J_{s,k}(X)\le CX^{\varepsilon}\Big(X^s+X^{2s-\frac{k(k+1)}{2}}\Big)\text{,}
\end{equation}
where
\[
J_{s,k}(X)
\coloneqq\big|\big\{(\boldsymbol{u},\boldsymbol{v})\in [X]^s\times[X]^s:\,
\sum_{i\in[s]}(u_i^j-v_i^j)=0
\text{ for every }j\in[k]\big\}\big|.
\]
\end{theorem}
\begin{proof} The case $k=1$ is trivial, the case $k=2$ is  classical, and the cases $k= 3$ and $k\ge 4$ are breakthrough results of Wooley \cite{W} and Bourgain-Demeter-Guth \cite{BDG}, respectively, see also \cite{Lillian}.
\end{proof}
A key mechanism of Vinogradov's method for estimating exponential sums is provided below.
\begin{proposition}\label{genericvin}
Let $s,k,N\in\mathbb{N}$ and let $\mathfrak{h}\in\mathcal{C}^{k+1}([N,3N]\to\mathbb{R})$. Then for every  $X\in[3, N^{1/2}]\cap\mathbb{N}$ and $I\subseteq [k]$ we have
\begin{equation}
\label{vinmethod}
\begin{split}
\Big|\frac{1}{N}\sum_{a=N+1}^{2N}e(\mathfrak h(a))\Big|\le& \frac{1}{N}\sum_{a=N+1}^{2N}J_{s,k}(X)^{\frac{1}{2s^2}}\Delta_I(a)^{\frac{1}{4s^2}} s^{\frac{k}{2s^2}}X^{\frac{k(k+1)}{4s^2}-\frac{1}{s}}\\
&+2\pi X^{2(k+1)}\sup_{\theta\in(N+1, 3N)}\bigg|\frac{\mathfrak h^{(k+1)}(\theta)}{(k+1)!}\bigg|+\frac{2X^2}{N}\text{,}
\end{split}
\end{equation}
where 
\begin{equation}
\label{DeltaIdef}
\Delta_I(a)\coloneq\big(10k\log(3sX)\big)^k\prod_{j\in I}\bigg(\bigg|\frac{\mathfrak h^{(j)}(a)}{j!}\bigg|+\bigg|\frac{\mathfrak h^{(j)}(a)}{j!} \bigg|^{-1}X^{-2j}\bigg).
\end{equation}
\end{proposition}
\begin{proof}
For a proof, we direct the reader to Section~3 in \cite{DMW} where all the generic steps of Vinogradov's method are provided and the estimate  \eqref{vinmethod} is  established. See also subsection~8.5 in \cite{IWKO} for an exposition of the method.
\end{proof}
\begin{lemma}\label{elementarybutdifficult}
Assume $f\colon \mathbb{T}\to\mathbb{R}$ is a smooth function such that $f'$ does not have a root of infinite order. Then $f$ has finitely many roots and for every $\delta>0$ and $C_1>0$ there exists $C_2=C_2(f,C_1,\delta)>0$ such that for every $N\ge C_2$ we have
\begin{equation}
\big\{x\in\mathbb{T}:\,|f(x)|<C_1N^{-\delta}\big\}=\bigcup_{\substack{x_0\in\mathbb{T}\\f(x_0)=0}}I_{x_0}\text{,}
\end{equation}
where $\{I_{x_0}:\, x_0\in\mathbb{T}\text{ with }f(x_0)=0\}$ is a collection of mutually disjoint intervals in $\mathbb{T}$ such that each $I_{x_0}$ contains $x_0$ and has length at most $C_2N^{-\delta/m_{x_0}}$, with $m_{x_0}\in\mathbb{N}$ being the order of the root, namely, 
\begin{equation}\label{mult}
m_{x_0}\coloneq \min\big(\{j\in\mathbb{N}:\,f^{(j)}(x_0)\neq 0\}\big)\text{.}
\end{equation}
In particular, if $f$ has no roots, then for every $N\ge C_2$ we have that $\big\{x\in\mathbb{T}:\,|f(x)|<C_1N^{-\delta}\big\}=\emptyset$, and in either case we have that there exist $\nu=\nu(f,\delta)>0$ and $\widetilde{C}_2=\widetilde{C}_2(f,C_1,\delta)>0$ such that for every $N\ge C_2$ we have
\begin{equation}\label{alltogethernew}
\lambda_{\mathbb{T}}\big(\big\{x\in\mathbb{T}:\,|f(x)|<C_1N^{-\delta}\big\}\big)\le \widetilde{C}_2N^{-\nu}\text{.}
\end{equation}
\end{lemma}
\begin{proof}
This lemma is elementary, but its complete proof is quite long. We provide all the details below but we encourage the reader to skip to the the beginning of the proof of Proposition~$\ref{minarcprop}$ in their first reading.

Firstly, we show that $f'$ has finitely many roots. To see this, let us identify $\mathbb{T}$ with $ [0,1)$ and assume for the sake of a contradiction that $f'$ has infinitely many roots $(x_n)_{n\in\mathbb{N}}$. Then by compactness and standard arguments, there exists a strictly monotone subsequence $(x_{n_k})_{k\in\mathbb{N}}$ and $x_0\in[0,1)$ such that $\lim_{k\to\infty}\|x_{n_k}- x_0\|=0$, and by continuity $f'(x_0)=0$. By the Mean Value Theorem applied to the intervals with endpoints $x_{n_k}$ and $x_{n_{k+1}}$, we deduce the existence of a strictly monotone sequence $(y_{n_k})_{k\in\mathbb{N}}$ such that $\|y_{n_k}-x_0\|\to 0$ and $f''(y_{n_k})=0$ for every $k\in\mathbb{N}$, and again, by the continuity of $f''$ we get that $f''(x_0)=0$. Inductively we may show that $f^{(j)}(x_0)=0$ for every $j\in\mathbb{N}$ and thus deduce the existence of a root of infinite order for $f'$, which is a contradiction. Thus $f'$ has finitely many roots.

A similar argument shows that $f$ has finitely many roots, since otherwise, by contradiction and using the Mean Value Theorem $f'$ would also have infinitely many roots. Also, it is easy to see that each of these finitely many roots must have finite order since every root of $f'$ has finite order, and thus \eqref{mult} is well-defined.

Secondly, we note that the set $S_f\coloneq \big\{|f(x)|:\,x\in\mathbb{T}\text{ is such that }f'(x)=0\text{ and }f(x)\neq 0\big\}$ is finite (and by periodicity and our assumption nonempty), and we therefore get that
\[
\min(S_f)\eqcolon c(f)>0\text{.}
\] 
Note that $c(f)/2$ is a positive real number smaller than any nonzero local minimum value of the function $|f(x)|$.

Fix $\delta>0$ and $C_1>0$, and note that there exists $C_2'=C_2'(f,C_1,\delta)>0$ such that  $C_1(C_2')^{-\delta}\le c(f)/2$, and thus for every $N\ge C_2'$ we have that
\begin{equation}\label{nolocal}
\big\{x\in\mathbb{T}:\,|f(x)|<C_1N^{-\delta}\big\}\cap \{x\in\mathbb{T}:\,f'(x)=0\text{ and }f(x)\neq 0\}=\emptyset\text{,}
\end{equation}
and, in particular, for $N\ge C_2'$ the first set above does not contain any local minimum or  local maximum $x$ of $f$ for which $f(x)\neq0$.

If $f$ has no roots, by the previous argument we get that for every $N\ge C_2'$ the set $\big\{x\in\mathbb{T}:\,|f(x)|<C_1N^{-\delta}\big\}=\emptyset$, since the minimum of $|f|$ is then nonzero and corresponds to a local maximum or local minimum of $f$. So let us assume from now on that $f$ has at least one root.

By the continuity of $f$, the set $\{x\in\mathbb{T}:\,|f(x)|<C_1N^{-\delta}\}$ can be written as a union of mutually disjoint at most countably many open intervals in $\mathbb{T}$ of maximal length $(I_s)_{s\in \mathcal{I}}$, that is
\[
\{x\in\mathbb{T}:\,|f(x)|<C_1N^{-\delta}\}=\bigcup_{s\in \mathcal{I}}I_s\text{.}
\] 
By the maximality of the intervals, any endpoint $x_{\text{endpt}}$ of $I_s$ must be such $|f(x_{\text{endpt}})|=C_1N^{-\delta}$, and thus for $N\ge C_2'$ every interval $I_s$ must contain a root of $f$, since if we assume not, then it contains a local maximum or a local minimum which is not a root, contradicting \eqref{nolocal}. Finally, every $I_s$ must have exactly one root, since if it has two, again it must contain a local maximum or minimum which is not a root, contradicting \eqref{nolocal}.

Thus, letting $\mathcal{R}_f\coloneq\{x_0\in\mathbb{T}:\,f(x_0)=0\}$ be the set of roots of $f$, we may reparametrize the family by choosing the unique root of every interval 
\[
\{x\in\mathbb{T}:\,|f(x)|<C_1N^{-\delta}\}=\bigcup_{x_0\in\mathcal{R}_f}I_{x_0}\text{,}
\]
and to complete the proof of the first assertion of the lemma it remains to bound the length of each $I_{x_0}$. Let 
\[
J_{x_0}\coloneq \big[x_0-\kappa_1(f,x_0),x_0+\kappa_2(f,x_0)\big]
\]
where $\kappa_1,\kappa_2$ are maximally chosen so that $f$ is monotone in $[x_0-\kappa_1,x_0]$ and $[x_0,x_0+\kappa_2]$. Note that $J_{x_0}$ does not depend on $N$, and that $I_{x_0}\subseteq J_{x_0}$, since $I_{x_0}$ does not have a local minimum or local maximum which is not a root. Also, note that the function $g\colon J_{x_0}\to \mathbb{R}$ given by  
\[
g(x)\coloneq\frac{f(x)}{(x-x_0)^{m_{x_0}}}
\]
is well-defined by L'H\^opital, it is never $0$ and it is continuous. Thus its absolute value has a nonzero minimum (independent of $N$), namely, there exists $c_{x_0}=c_{x_0}(f)\in(0,1)$ such that $|f(x)|\ge c_{x_{0}}(x-x_0)^{m_0}$ for every $x\in J_{x_0}$, and in particular, for every $x\in I_{x_0}$. Thus if $I_{x_0}=(x_1,x_2)$, we have  for $i\in[2]$ that
\[
|x_i-x_0|\le \bigg(\frac{2C_1}{c_{x_0}}\bigg)^{1/m_{x_0}} N^{-\delta/m_{x_0}}  
\]
since otherwise $|f(x_i)|\ge c_{x_0}|x_i-x_0|^{m_{x_0}}> 2C_1N^{-\delta}$, which is a contradiction. Thus 
\[
|I_{x_0}|\le2\bigg(\frac{2C_1}{c_{x_0}}\bigg)^{1/m_{x_0}} N^{-\delta/m_{x_0}}\text{,}  
\]
and letting $C_2\coloneq \max\Big(\big\{C_2'\big\}\cup\Big\{2\Big(\frac{2C_1}{c_{x_0}}\Big)^{\frac{1}{m_{x_0}}}:\,x_0\in\mathcal{R}_f\Big\}\Big)>0$ yields the first assertion of the lemma.

The second assertion of the lemma follows easily from the first. In the case where $\mathcal{R}_f=\emptyset$ the estimate \eqref{alltogethernew} holds for any choice of $\widetilde{C}_2$ and $\nu$ since $\big\{x\in\mathbb{T}:\,|f(x)|<C_1N^{-\delta}\big\}=\emptyset$ for $N\ge C_2$. In the nontrivial case when $\mathcal{R}_f\neq\emptyset$ , we take $\nu=\delta/m_*>0$, where $m_*\coloneq \max(\{m_{x_0}:\,x_0\in\mathcal{R}_f\})$, and $\widetilde{C}_2\coloneq C_2|\mathcal{R}_f|$, and we immediately see that for every $N\ge C_2$ we get
\begin{equation}
\lambda_{\mathbb{T}}\big(\big\{x\in\mathbb{T}:\,|f(x)|<C_1N^{-\delta}\big\}\big)=\sum_{x_0\in\mathcal{R}_f}|I_{x_0}|\le \widetilde{C}_2N^{-\nu}\text{,}\quad\text{and the proof is complete.}
\end{equation}
\end{proof}
\subsection{Main minor arc estimate.} We are now ready to give the proof of Proposition~$\ref{minarcprop}$.

\begin{proof}[Proof of Proposition~$\ref{minarcprop}$]
We fix $\psi,\alpha,P,\delta$ as in the hypotheses of the theorem and all implicit constants may depend on them. In particular, they are allowed to depend on $\omega\coloneq\omega(\alpha)$ and $d\coloneq\deg(P)$. It clearly suffices to prove the result for $N\in\mathbb{N}$ and $N\gtrsim 1$.
\smallskip
\paragraph{\textbf{Step 1: Diophantine approximation of $\alpha$ and passing to arithmetic progressions.}} We begin by carefully choosing some parameters. We fix $\rho=\rho(\omega,d)$ such that  
\begin{equation}\label{rho1}
1-\frac{1}{\omega}<\rho<\frac{d+1}{d+\omega}\text{.}
\end{equation}
To see why this is possible, we note that
\begin{equation}
\begin{split}
1-\frac{1}{\omega}<\frac{d+1}{d+\omega}\iff& \omega(d+\omega)-\frac{\omega(d+\omega)}{\omega}<\frac{\omega(d+\omega)(d+1)}{d+\omega}
\\
\iff& \omega d+\omega^2-d-\omega<\omega d+\omega\iff \omega^2-2\omega+1<d+1
\\
\iff& 1-\sqrt{d+1}<\omega<1+\sqrt{d+1}\text{,}
\end{split}
\end{equation} 
which is true, since by our assumption we have that $2\le \omega<1+\sqrt{d+1}$. We may take $\rho=\rho(\omega,d)$ as the center of the nonempty interval $\big(1-\frac{1}{\omega},\frac{d+1}{d+\omega}\big)$. We immediately get that $\rho\in(1/2,1)$ since $\omega\ge 2$.

For every $N\in\mathbb{N}$ by Dirichlet's approximation theorem there exists 
a pair of integers $(p,q)=(p(N),q(N))\in\mathbb{Z}\times\mathbb{N}$ with $\gcd(p,q)=1$ and $q\le N^{\rho}$, such that
\begin{equation}\label{Dirapprox}
\big|\alpha-p/q\big|\le \frac{1}{qN^{\rho}}\text{.}
\end{equation}
Since $\rho$ is fixed so that the inequalities in \eqref{rho1} are strict, we can fix $\omega'=\omega'(\omega,d)>\omega$ such that 
\begin{equation}\label{rho'1}
1-\frac{1}{\omega'}<\rho<\frac{d+1}{d+\omega'}\text{.}
\end{equation}

By Lemma~$\ref{expbreath}$ there exists $c=c(\alpha,\omega')\in(0,1)$ such that
\begin{equation}\label{irexponentbound}
\big|\alpha-a/b\big|\ge \frac{c}{b^{\omega'}}\quad\text{for every $(a,b)\in\mathbb{Z}\times\mathbb{N}$ with $\gcd(a,b)=1$. }
\end{equation}
Combining $\eqref{Dirapprox}$ and \eqref{irexponentbound} we immediately get that
\begin{equation}\label{qlowerbound}
q\gtrsim N^{\frac{\rho}{\omega'-1}}\text{,}
\end{equation}
since
\[
\frac{1}{qN^{\rho}}\ge\frac{c}{q^{\omega'}}\Longrightarrow q^{\omega'-1}\gtrsim N^{\rho}\Longrightarrow q\gtrsim\footnote{Implicit constants may depend on our fixed parameters $\psi,\alpha,P,\delta$ without further mention. This means that whenever we introduce additional parameters \textit{depending only on} $\psi,\alpha,P,\delta$, for example $\rho$ or $\omega'$, we may treat them as fixed as well, and all implicit constants may depend on them.} N^{\frac{\rho}{\omega'-1}}\text{.}
\]
Letting $\theta=\theta(N)\coloneq \alpha-\frac{p}{q}$, we collect the following upper bounds
\begin{equation}\label{upperboundsfortheta}
|\theta|\le \frac{1}{qN^{\rho}}\lesssim N^{-\rho-\frac{\rho}{\omega'-1}}\text{,}\quad q|\theta|\le N^{-\rho}\text{,}\quad N|\theta|\lesssim N^{1-\rho-\frac{\rho}{\omega'-1}}\text{,}
\end{equation}
as well as the lower bound
\begin{equation}\label{lowerboundfortheta}
|\theta|\gtrsim q^{-\omega'}\ge N^{-\rho\omega'}\text{.}
\end{equation}
Crucially, as will become apparent later, since $\rho$ and $\omega'$ are fixed so that $\rho>1-\frac{1}{\omega'}$, we get that the final upper bound in \eqref{upperboundsfortheta}  yields a power saving in $N$, since
\begin{equation}\label{taubounds}
-1<1-\rho-\frac{\rho}{\omega'-1}<1-(1-1/\omega')-\frac{1-1/\omega'}{\omega'-1}=\frac{1}{\omega'}-\frac{1}{\omega'-1}+\frac{1}{\omega'(\omega'-1)}=0\text{.}
\end{equation}
Thus letting
\begin{equation}\label{taudef}
\tau\coloneq \rho+\frac{\rho}{\omega'-1}-1\text{,}
\end{equation} 
we see by \eqref{taubounds} that $\tau\in(0,1)$ and the final upper bound in \eqref{upperboundsfortheta} becomes
\begin{equation}\label{thetauseful}
N|\theta|\lesssim N^{-\tau}\text{.}
\end{equation}

We now split the summation over classes modulo $q$ and estimate as follows
\begin{equation}
\begin{split}
\Big|\sum_{n=N+1}^{2N}&e(\xi P(n)\psi(\alpha n))\Big|=\Big|\sum_{r=0}^{q-1}\sum_{\frac{N+1-r}{q}\le m\le \frac{2N-r}{q}}e\big(\xi P(qm+r)\psi(\alpha(qm+r))\big)\Big|
\\
\le&\Big|\sum_{r=0}^{q-1}\sum_{\frac{N}{q}\le m\le \frac{2N}{q}}e\big(\xi P(qm+r)\psi(\alpha(qm+r))\big)\Big|
+\sum_{r=0}^{q-1}\sum_{\frac{2N-r}{q}< m\le \frac{2N}{q}}1+\sum_{r=0}^{q-1}\sum_{\frac{N+1-r}{q}\le m< \frac{N}{q}}1\text{.}
\end{split}
\end{equation}
The second summand can be immediately estimated
\begin{equation}
\sum_{r=0}^{q-1}\sum_{\frac{2N-r}{q}< m\le \frac{2N}{q}}1+\sum_{r=0}^{q-1}\sum_{\frac{N+1-r}{q}\le m< \frac{N}{q}}1\le2\sum_{r=0}^{q-1}\bigg(\frac{r}{q}+1\bigg)\le 4q\le 4N^{\rho}\text{,}
\end{equation}
and since $\rho$ is strictly smaller than $1$ we may focus on the first summand. Using the periodicity of $\psi$ as in \eqref{psiperiod}we obtain the following
\begin{equation}\label{S1outcome}
\begin{split}
\Big|\sum_{n=N+1}^{2N}e(\xi P(n)\psi(\alpha n))\Big|\le\sum_{r=0}^{q-1}\bigg|\sum_{\frac{N}{q}\le m\le \frac{2N}{q}}e\Big(\xi P(qm+r)\psi\big(\theta(qm+r)+rp/q\big)\Big)\bigg|+4N^{\rho}\text{.}
\end{split}
\end{equation}
\smallskip
\paragraph{\textbf{Step 2: Derivative estimates in preparation for Vinogradov's method.}} We wish to apply the Vinogradov method for estimating exponential sums for ``most'' of the inner sums above, and the trivial estimate for the rest. The determining factor of whether an inner exponential sum falls in the former or the latter category will be whether we have appropriate control for a lot of its derivatives, and to make this distinction precise, we carefully calculate the derivatives of the function in the phase with respect to $m$. For every $x\in\mathbb{R}$ we define 
\begin{equation}
F(x)=F(x;N,\xi,r)\coloneq \xi P(qx+r)\psi\big(\theta(qx+r)+rp/q\big)\text{.} 
\end{equation}
We remind the reader that although we avoid indicating so, the values of $p,q,\theta$ are functions of $N$. For every $x\in\mathbb{R}$ and $j\ge d$ we have
\begin{equation}
\begin{split}
F^{(j)}(x)=&\sum_{i=0}^j\binom{j}{i}\big(\xi P(qx+r)\big)^{(i)}\Big(\psi\big(\theta(q x+r)+rp/q\big)\Big)^{(j-i)}
\\
=&\xi \sum_{i=0}^d\binom{j}{i}q^iP^{(i)}(qx+r)\theta^{j-i}q^{j-i}\psi^{(j-i)}\big(\theta(q x+r)+rp/q\big)
\\
=&\xi q^j \sum_{i=0}^d\binom{j}{i}P^{(i)}(qx+r)\theta^{j-i}\psi^{(j-i)}\big(\theta(q x+r)+rp/q\big)\text{,}
\end{split}
\end{equation}
where all the derivatives are taken with respect to $x$.

To keep our implicit constants depending only on the choices $\alpha,P,\psi,\delta$, it will be useful to fix
\begin{equation}\label{s_0choice}
d_0\coloneq\max\bigg(\lfloor d\omega'+4(\omega'-1)\rfloor+2,\bigg\lfloor\frac{d+2}{\tau}\bigg\rfloor+1+d\bigg)\text{,}
\end{equation}
and work only with $j\in[d,d_0]\cap\mathbb{N}$. The reader should ignore at this stage the precise choice, since its utility is evident only in the last steps of the proof, where $d_0-1$ will be the order of the Taylor approximation in Vinogradov's method, see \eqref{choicesvinogradov}. The important aspect at this point is that it is fixed, in the sense that all implicit constants can depend on $d_0$. Let us note, however, that
\begin{equation}\label{lowerboundd0}
d_0\ge 2d+3\text{,}\quad\text{since $\tau\in(0,1)$, and thus $[d,d_0]\cap\mathbb{N}\neq \emptyset$.}
\end{equation} 
 
For every $r\in\{0,\dotsc,q-1\}$,  $x\in[N/q,4N/q]$ and $j\in[d,d_0]\cap\mathbb{N}$ we have
\begin{equation}\label{upperboundderivative}
\begin{split}
|F^{(j)}(x)|\le&|\xi| q^j \sum_{i=0}^d\binom{j}{i}|P^{(i)}(qx+r)||\theta|^{j-i}\|\psi^{(j-i)}\|_{L^{\infty}}
\\
\lesssim&|\xi| q^j\sum_{i=0}^d\binom{j}{i}N^{d-i}|\theta|^{j-i}
\lesssim|\xi| q^j\sum_{i=0}^d\binom{j}{i}|\theta|^{-d+i}|\theta|^{j-i}\lesssim |\xi| q^j|\theta|^{j-d}
\text{,}
\end{split}
\end{equation}
where we have used that $N\lesssim |\theta|^{-1}$ by \eqref{thetauseful}, as well as the fact that $\|\psi^{(l)}\|_{L^{\infty}}<\infty$ for every $l\in\mathbb{N}_0$ since $\psi$ is periodic and smooth.

On the other hand, for every $r\in\{0,\dotsc,q-1\}$, $x\in[N/q,4N/q]$ and $j\in[ d,d_0]\cap\mathbb{N}$ we have that
\begin{equation}\label{pickmainterm}
\begin{split}
|F^{(j)}(x)|=&|\xi| q^j \bigg|\sum_{i=0}^d\binom{j}{i}P^{(i)}(qx+r)\theta^{j-i}\psi^{(j-i)}\big(\theta(q x+r)+rp/q\big)\bigg|
\\
\ge&|\xi| q^j\bigg(\binom{j}{d}|P^{(d)}(qx+r)||\theta|^{j-d}\big|\psi^{(j-d)}\big(\theta(q x+r)+rp/q\big)\big|
\\
&-\bigg|\sum_{i=0}^{d-1}\binom{j}{i}P^{(i)}(qx+r)\theta^{j-i}\psi^{(j-i)}\big(\theta(q x+r)+rp/q\big)\bigg|\bigg)
\text{,}
\end{split}
\end{equation}
and it is not difficult to estimate the second summand above as follows
\begin{equation}
\begin{split}
\bigg|\sum_{i=0}^{d-1}\binom{j}{i}P^{(i)}(qx+r)\theta^{j-i}\psi^{(j-i)}\big(\theta(q x+r)+rp/q\big)\bigg|
\lesssim \sum_{i=0}^{d-1}N^{d-i}|\theta|^{j-i}\lesssim |\theta|^{j-d}N^{-\tau}\text{,}
\end{split}
\end{equation}
since for every $0\le i\le d-1$ we have $N^{d-i}|\theta|^{j-i}\lesssim |\theta|^{j-d}N^{-\tau}$ because $N^{d-i}|\theta|^{d-i}\lesssim N^{-(d-i)\tau}\le N^{-\tau}$.

For the first summand  we have that
\begin{equation}
\binom{j}{d}|P^{(d)}(qx+r)||\theta|^{j-d}\big|\psi^{(j-d)}\big(\theta(q x+r)+rp/q\big)\big|\gtrsim|\theta|^{j-d}\big|\psi^{(j-d)}\big(\theta(q x+r)+rp/q\big)\big|\text{.}
\end{equation}
A careful analysis is now necessary and we abandon momentarily the asymptotic notation `` $\lesssim$ ''. We have shown that there exist two positive constants $c_1=c_1(\alpha,P,\psi,\delta)$ and $C_2=C_2(\alpha,P,\psi,\delta)$  such that for every $x\in[N/q,4N/q]$, $j\in[d,d_0]\cap\mathbb{N}$ and $r\in\{0,\dotsc,q-1\}$ we have
\begin{equation}\label{upperfirst}
\binom{j}{d}|P^{(d)}(qx+r)||\theta|^{j-d}\big|\psi^{(j-d)}\big(\theta(q x+r)+rp/q\big)\big|\ge c_1 |\theta|^{j-d}\big|\psi^{(j-d)}\big(\theta(q x+r)+rp/q\big)\big|\text{,}
\end{equation}
and 
\begin{equation}\label{lowersecond}
\bigg|\sum_{i=0}^{d-1}\binom{j}{i}P^{(i)}(qx+r)\theta^{j-i}\psi^{(j-i)}\big(\theta(q x+r)+rp/q\big)\bigg|\le C_2|\theta|^{j-d}N^{-\tau}\text{.}
\end{equation}
We fix a small positive number $\sigma$ as follows
\begin{equation}\label{tau'choice}
\sigma=\sigma(\alpha,P,\psi,\delta)
\coloneq \min\bigg(\frac{\tau}{2},\frac{1-\rho}{4},\frac{\delta}{4},\frac{(d+1)-\rho(d+\omega')}{8}\bigg)>0\text{,}
\end{equation}
and, similarly to $d_0$, the usefulness of this choice (and specifically of the smallness of the last three arguments in the $\min$) is visible only in the last step of Vinogradov's method. 
The first three arguments are clearly positive, and for the last one we note that
\[
(d+1)-\rho(d+\omega')>0\iff\rho<\frac{d+1}{d+\omega'}\text{,}\quad\text{which holds, by \eqref{rho'1}.
}
\]
We let 
\begin{equation}\label{mainsetM}
\mathcal{M}(N)=\bigg\{r\in\{0,\dotsc,q-1\}:\,\min_{\substack{x\in[N/q,4N/q]\text{,}\\0\le i\le d_0}}\big|\psi^{(i)}\big(\theta(qx+r)+rp/q\big)\big|\ge \frac{2C_2}{c_1}N^{-\sigma}\bigg\}\text{,}
\end{equation}
and we let $\mathcal{E}(N)$ be its complement, i.e.: 
\begin{equation}\label{seterror}
\mathcal{E}(N)=\bigg\{r\in\{0,\dotsc,q-1\}:\,\min_{\substack{x\in[N/q,4N/q]\text{,}\\0\le i\le d_0}}\big|\psi^{(i)}\big(\theta(qx+r)+rp/q\big)\big|< \frac{2C_2}{c_1}N^{-\sigma}\bigg\}\text{.}
\end{equation}
This is precisely the splitting that was mentioned in the beginning of Step~2, see also \eqref{Step2end}. The advantage, is that for every $r\in\mathcal{M}(N)$, every $x\in[N/q,4N/q]$, and every $j\in[d,d_0]\cap\mathbb{N}$, by \eqref{pickmainterm}, \eqref{upperfirst} and \eqref{lowersecond}, we have that
\begin{equation}
\begin{split}
|F^{(j)}(x)|\ge|\xi|q^j\Big(c_1 |\theta|^{j-d}\big|\psi^{(j-d)}\big(\theta(q x+r)+rp/q\big)\big|-C_2|\theta|^{j-d}N^{-\tau}\Big)
\\
\ge |\xi|q^j\bigg(c_1 |\theta|^{j-d}\frac{2C_2}{c_1}N^{-\sigma}-C_2|\theta|^{j-d}N^{-\sigma}\bigg)=C_2|\xi|q^j |\theta|^{j-d}N^{-\sigma}\text{.}
\end{split}
\end{equation}
Thus, returning to our previous asymptotic notation and using \eqref{upperboundderivative}, we have shown that for every $r\in\mathcal{M}(N)$, $x\in[N/q,4N/q]$ and $j\in[d,d_0]\cap\mathbb{N}$ we have
\begin{equation}\label{derivativefull}
N^{-\sigma}\cdot \big(|\xi|q^j |\theta|^{j-d}\big)\lesssim |F^{(j)}(x)|\lesssim |\xi|q^j |\theta|^{j-d}\text{,}
\end{equation}
which will be the input used in the Vinogradov method.

Returning to \eqref{S1outcome} we therefore split the sum as follows
\begin{equation}\label{Step2end}
\begin{split}
\Big|\sum_{n=N+1}^{2N}e(\xi P(n)\psi(\alpha n))\Big|\le&\sum_{r=0}^{q-1}\bigg|\sum_{\frac{N}{q}\le m\le \frac{2N}{q}}e\Big(\xi P(qm+r)\psi\big(\theta(qm+r)+rp/q\big)\Big)\bigg|+4N^{\rho}
\\
\lesssim&\sum_{r\in\mathcal{M}(N)}\Big|\sum_{\frac{N}{q}\le m\le \frac{2N}{q}}e\big(F(m;N,\xi,r)\big)\Big|+N^{\rho}+\sum_{r\in\mathcal{E}(N)}Nq^{-1}
\\
=&\sum_{r\in\mathcal{M}(N)}\Big|\sum_{\frac{N}{q}\le m\le \frac{2N}{q}}e\big(F(m;N,\xi,r)\big)\Big|+N^{\rho}+N\cdot\frac{|\mathcal{E}(N)|}{q}
\text{,}
\end{split}
\end{equation}
and we have good control of a lot of derivatives of the phase in the exponential sums above, see \eqref{derivativefull}. In what follows, we prove that the last summand above is appropriately decaying, and then we perform the Vinogradov method to treat the exponential sums in the first summand.
\smallskip
\paragraph{\textbf{Step 3: Estimating the number of exceptional classes modulo $q$.}}
In this step we bound the last summand in the last line of \eqref{Step2end}. Firstly, note that
\begin{equation}
\mathcal{E}(N)\subseteq\bigcup_{i=0}^{d_0}\mathcal{E}_i(N)\text{,}
\end{equation}
where 
\begin{equation}
\mathcal{E}_i(N)\coloneq \bigg\{r\in\{0,\dotsc,q-1\}:\,\min_{x\in[N/q,4N/q]}\big|\psi^{(i)}\big(\theta(qx+r)+rp/q\big)\big|< \frac{2C_2}{c_1}N^{-\sigma}\bigg\}\text{.}
\end{equation}
It is not difficult to see that for $N\gtrsim 1$ and for every $0\le i\le d_0$ we have
\begin{equation}\label{removex}
\mathcal{E}_i(N)\subseteq\bigg\{r\in\{0,\dotsc,q-1\}:\,|\psi^{(i)}(rp/q)|< \frac{4C_2}{c_1}N^{-\sigma}\bigg\}\eqcolon \mathcal{E}_{i,1}(N)\text{.}
\end{equation}
To see this, we note that for every $r\in\mathcal{E}_i(N)$ we have that there exists $x\in[N/q,4N/q]$ such that
\begin{equation}
\big|\psi^{(i)}\big(\theta(qx+r)+rp/q\big)\big|< \frac{2C_2}{c_1}N^{-\sigma}\text{,}
\end{equation} 
but then for every $0\le i\le d_0$ and $N\gtrsim 1 $ we have
\begin{equation}
\begin{split}
|\psi^{(i)}(rp/q)|&\le |\psi^{(i)}(rp/q)-\psi^{(i)}\big(\theta(qx+r)+rp/q\big)\big|+\big|\psi^{(i)}\big(\theta(qx+r)+rp/q\big)\big|
\\
&\le \|\psi^{(i+1)}\|_{L^{\infty}}|\theta|5N+\frac{2C_2}{c_1}N^{-\sigma}\le \frac{2C_2}{c_1}N^{-\tau+(\tau-\sigma)/2}+\frac{2C_2}{c_1}N^{-\sigma}\le\frac{4C_2}{c_1}N^{-\sigma} \text{,}
\end{split}
\end{equation}
where for the last estimate we note that 
\[
-\tau+(\tau-\sigma)/2<-\sigma\iff(\tau-\sigma)/2<\tau-\sigma\iff \tau-\sigma>0\text{,}
\]
which is true, since $\sigma<\tau$, and for the penultimate estimate we note that for $N\gtrsim1$ we have that
\[
\|\psi^{(i+1)}\|_{L^{\infty}}|\theta|5N\le \frac{2C_2}{c_1}N^{-\tau+(\tau-\sigma)/2}\text{,}
\]
since
\[
\|\psi^{(i+1)}\|_{L^{\infty}}|\theta|5N\le \frac{2C_2}{c_1}N^{-\tau+(\tau-\sigma)/2}\iff N|\theta|\le \frac{2C_2}{5c_1\|\psi^{(i+1)}\|_{L^{\infty}}}N^{-\tau+(\tau-\sigma)/2}\text{,}
\] 
and the last estimate clearly holds for $N\gtrsim1$ since $N|\theta|\lesssim N^{-\tau}$, and $\tau-\sigma>0$. Note that $\|\psi^{(i+1)}\|_{L^{\infty}}>0$ by our assumption. Thus for $N\gtrsim1$ and for every $r\in\mathcal{E}_i(N)$ we have that $
|\psi^{(i)}(rp/q)|< \frac{4C_2}{c_1}N^{-\sigma}$, and \eqref{removex} is justified.

Now we note that  
\begin{equation}
|\mathcal{E}_{i,1}(N)|=|\mathcal{E}_{i,2}(N)|\text{,}
\end{equation}
where 
\begin{equation}
\mathcal{E}_{i,2}(N)\coloneq\bigg\{r\in\{0,\dotsc,q-1\}:\,|\psi^{(i)}(r/q)|< \frac{4C_2}{c_1}N^{-\sigma}\bigg\}\text{.}
\end{equation}
This is immediate by noting that the map $\mathbb{Z}\ni r\to\psi^{(i)}(r/q)$ depends only on the value of $r\Mod{q}$ by the periodicity of $\psi^{(i)}$, and that the map $\mathbb{Z}/(q\mathbb{Z})\ni r\to pr\in\mathbb{Z}/(q\mathbb{Z})$ is a bijection, since $\gcd(p,q)=1$. 

It is natural to consider 
\begin{equation}
B_i(N)\coloneq\bigg\{x\in\mathbb{T}:\,|\psi^{(i)}(x)|< \frac{4C_2}{c_1}N^{-\sigma}\bigg\}
\end{equation}
so that
\[
\mathcal{E}_{i,2}(N)=\big\{r\in\{0,\dotsc,q-1\}:\,r/q\in B_i(N)\big\}\text{.}
\]
We may now use Lemma~$\ref{elementarybutdifficult}$ to conclude. More precisely, the aforementioned lemma guarantees that there exists $\sigma'>0$ such that for $N\gtrsim1$ we have that $B_i(N)$ is a finite union of intervals, the number of which is $\lesssim 1$, and each interval has length at most $\lesssim N^{-\sigma'}$. The number of $r$'s in $\{0,\dotsc,q-1\}$ such that $r/q$ belongs in one of these interval is bounded by $\lesssim qN^{-\sigma'}+1$, and thus
\[
|\mathcal{E}(N)|\le \sum_{i=0}^{d_0}|\mathcal{E}_{i}(N)|\le\sum_{i=0}^{d_0}|\mathcal{E}_{i,1}(N)|=\sum_{i=0}^{d_0}|\mathcal{E}_{i,2}(N)|\lesssim\sum_{i=0}^{d_0} \big(qN^{-\sigma'}+1\big)\lesssim qN^{-\sigma'}+1\text{.}
\] 
Thus, the final summand in \eqref{Step2end} can now be estimated as follows
\[
N\cdot\frac{|\mathcal{E}(N)|}{q}\lesssim N\big(N^{-\sigma'}+q^{-1}\big)\lesssim N^{1-\sigma'}+N/q \lesssim N^{1-\sigma'}+N^{1-\frac{\rho}{\omega'-1}}\text{,}
\]
and returning to \eqref{Step2end}, we have shown that there exists $\sigma'>0$ such that
\begin{equation}\label{Step3end}
\begin{split}
\Big|\sum_{n=N+1}^{2N}e(\xi P(n)\psi(\alpha n))\Big|
\lesssim \sum_{r\in\mathcal{M}(N)}\Big|\sum_{\frac{N}{q}\le m\le \frac{2N}{q}}e\big(F(m;N,\xi,r)\big)\Big|+N^{\rho}+N^{1-\sigma'}+N^{1-\frac{\rho}{\omega'-1}}
\text{.}
\end{split}
\end{equation}
Since $\max\big(\rho,1-\sigma',1-\frac{\rho}{\omega'-1}\big)<1$, we may focus on the first summand above.
\smallskip
\paragraph{\textbf{Step 4: Employing the Vinogradov method.}} 
Let for convenience $M\coloneq \lfloor N/q\rfloor$. We may focus on the exponential sums along the range $[M+1,2M]$ since  
\begin{equation}
\begin{split}
\sum_{r\in\mathcal{M}(N)}\Big|\sum_{\frac{N}{q}\le m\le \frac{2N}{q}}e\big(F(m;N,\xi,r)\big)\Big|\le&\sum_{r\in\mathcal{M}(N)}\Big|\sum_{ m=M+1}^{2M}e\big(F(m;N,\xi,r)\big)\Big|+\sum_{r\in\mathcal{M}(N)}4
\\
\le&\sum_{r\in\mathcal{M}(N)}\Big|\sum_{ m=M+1}^{2M}e\big(F(m;N,\xi,r)\big)\Big|+4N^{\rho}\text{,}
\end{split}
\end{equation}
where we used the trivial estimate for the boundary summands, as well as the estimate $|\mathcal{M}(N)|\le q\le N^{\rho}$. 

We fix $r\in\mathcal{M}(N)$, apply Proposition~$\ref{genericvin}$, and obtain that for every $s,k\in\mathbb{N}$, $X\in[3,M^{1/2}]\cap\mathbb{N}$ and $I\subseteq [k]$, we have
\begin{equation}
\label{vinmethodnew}
\begin{split}
\Big|\sum_{a=M+1}^{2M}e(F(a;N,\xi,r))\Big|\le& \sum_{a=M+1}^{2M}J_{s,k}(X)^{\frac{1}{2s^2}}\Delta_I(a)^{\frac{1}{4s^2}} s^{\frac{k}{2s^2}}X^{\frac{k(k+1)}{4s^2}-\frac{1}{s}}\\
&+2\pi MX^{2(k+1)}\sup_{\theta\in(M+1, 3M)}\bigg|\frac{ F^{(k+1)}(\theta)}{(k+1)!}\bigg|+2X^2\text{,}
\end{split}
\end{equation}
where 
\begin{equation}
\label{DeltaIdefnew}
\Delta_I(a)\coloneq\big(10k\log(3sX)\big)^k\prod_{j\in I}\bigg(\bigg|\frac{F^{(j)}(a)}{j!}\bigg|+\bigg|\frac{ F^{(j)}(a)}{j!}\bigg|^{-1}X^{-2j}\bigg).
\end{equation}
We apply this for the following choices of parameters 
\begin{equation}\label{choicesvinogradov}
k\coloneq d_0-1\text{,}\quad s\coloneq \frac{k(k+1)}{2}\text{,}\quad X\coloneq \big\lfloor M^{1/2}N^{-\frac{\sigma}{2d_0}}\big\rfloor\text{.}
\end{equation} 
By \eqref{lowerboundd0} and \eqref{tau'choice}we get that $k\ge 1$ and $X\ge \big(\frac{N}{q}-1\big)^{1/2}N^{-\frac{\sigma}{2d_0}}-1\gtrsim N^{\frac{1-\rho-\sigma}{2}}-1\gtrsim N^{\sigma}-1$, respectively, and thus for $N\gtrsim1$ we do have that $X\ge 3$.

We are ready to treat the last two summands in the right side of \eqref{vinmethodnew}. One the one hand, the last one can be readily estimated since
\begin{equation}\label{3rdsum}
2X^2\le \frac{2N}{q}\cdot N^{-\frac{\sigma}{d_0}}=\frac{2N^{1-\frac{\sigma}{d_0}}}{q}\text{.}
\end{equation}
On the other hand, for the second summand we use the estimate \eqref{derivativefull}, the above estimate, and the fact that $k+1=d_0$ to estimate as follows
\begin{equation}\label{2ndsum}
\begin{split}
2\pi MX^{2(k+1)}\sup_{\theta\in(M+1, 3M)}\bigg|\frac{ F^{(k+1)}(\theta)}{(k+1)!}\bigg|\lesssim \frac{N}{q}\cdot\frac{N^{d_0-\sigma}}{q^{d_0}} \cdot|\xi|q^{d_0} |\theta|^{d_0-d}
\\
\le\frac{N}{q}\cdot N^{d_0-\sigma} \cdot N |\theta|^{d_0-d}
=\frac{1}{q}N^{2+d-\sigma} (N|\theta|)^{d_0-d}\lesssim \frac{1}{q}N^{-\sigma+(d+2)-(d_0-d)\tau}\text{,}
\end{split}
\end{equation}
where we have used the fact that $|\xi|\le N$. Taking into account \eqref{s_0choice}, we get that $d_0\ge \frac{d+2}{\tau}+d$, and thus
\[
\tau(d_0-d)\ge d+2\text{.}
\]
Therefore, we may proceed with our estimation as follows
\begin{equation}\label{2ndsumnew}
\begin{split}
2\pi MX^{2(k+1)}\sup_{\theta\in(M+1, 3M)}\bigg|\frac{ F^{(k+1)}(\theta)}{(k+1)!}\bigg|\lesssim \frac{1}{q}N^{(2+d-\sigma)-(d_0-d)\tau}\le \frac{1}{q}N^{-\sigma}\lesssim 1\text{.}
\end{split}
\end{equation}
Using \eqref{2ndsumnew}and \eqref{3rdsum} in \eqref{vinmethodnew} we get that for every $r\in\mathcal{M}(N)$ we have
\begin{equation}
\label{vinmethodnewnew}
\begin{split}
\Big|\sum_{a=M+1}^{2M}e(F(a;N,\xi,r))\Big|\lesssim & \bigg(\sum_{a=M+1}^{2M}J_{s,k}(X)^{\frac{1}{2s^2}}\Delta_I(a)^{\frac{1}{4s^2}} s^{\frac{k}{2s^2}}X^{\frac{k(k+1)}{4s^2}-\frac{1}{s}}\bigg)+1+\frac{N^{1-\frac{\sigma}{d_0}}}{q}\text{,}
\end{split}
\end{equation}
so let us focus on the first summand.

The first step would be to apply Theorem~$\ref{VMVTf}$ for our choice of natural numbers $s,k$ and for $\varepsilon=\frac{\sigma}{100}$ to obtain that
\[
J_{s,k}(X)\lesssim X^{\frac{\sigma}{100}}\Big(X^s+X^{2s-\frac{k(k+1)}{2}}\Big)=2X^{s+\frac{\sigma}{100}}\text{.}
\]
Thus, the first summand becomes
\begin{equation}
\begin{split}
\sum_{a=M+1}^{2M}J_{s,k}(X)^{\frac{1}{2s^2}}\Delta_I(a)^{\frac{1}{4s^2}} s^{\frac{k}{2s^2}}X^{\frac{k(k+1)}{4s^2}-\frac{1}{s}}\lesssim \sum_{a=M+1}^{2M}X^{\frac{1}{2s}+\frac{\sigma}{200s^2}}\Delta_I(a)^{\frac{1}{4s^2}} X^{-\frac{1}{2s}}
\\
\lesssim N^{\frac{\sigma}{200s^2}}\sum_{a=M+1}^{2M}\Delta_I(a)^{\frac{1}{4s^2}}\text{,}
\end{split}
\end{equation}
where we note that
\[
\frac{k(k+1)}{4s^2}-\frac{1}{s}=\frac{s}{2s^2}-\frac{1}{s}=-\frac{1}{2s}\text{,}
\] 
and we remind the reader that $s,k$ depend only on $\alpha,P,\psi,\delta$, and our implicit constants may therefore depend on them as well.

Combining this estimate with \eqref{vinmethodnewnew} we obtain
\begin{equation}
\label{vinmethodnewnewnew}
\begin{split}
\Big|\sum_{a=M+1}^{2M}e(F(a;N,\xi,r))\Big|\lesssim 1+\frac{N^{1-\frac{\sigma}{d_0}}}{q}+N^{\frac{\sigma}{200s^2}}\sum_{a=M+1}^{2M}\Delta_I(a)^{\frac{1}{4s^2}}\text{,}
\end{split}
\end{equation}
and returning to \eqref{Step3end} we obtain for $N\gtrsim1$ that
\begin{equation}\label{Step4end}
\begin{split}
\Big|\sum_{n=N+1}^{2N}e(\xi P(n)\psi(\alpha n))\Big|
\lesssim \sum_{r\in\mathcal{M}(N)}\Big|\sum_{\frac{N}{q}\le m\le \frac{2N}{q}}e\big(F(m;N,\xi,r)\big)\Big|+N^{\rho}+N^{1-\sigma'}+N^{1-\frac{\rho}{\omega'-1}}
\\
\lesssim q+N^{1-\frac{\sigma}{d_0}}+N^{\frac{\sigma}{200s^2}}\sum_{r\in\mathcal{M}(N)}\sum_{a=M+1}^{2M}\Delta_I(a)^{\frac{1}{4s^2}}+N^{\rho}+N^{1-\sigma'}+N^{1-\frac{\rho}{\omega'-1}}
\\
\lesssim N^{\rho}+N^{1-\frac{\sigma}{d_0}}+N^{1-\sigma'}+N^{1-\frac{\rho}{\omega'-1}}+N^{\frac{\sigma}{200s^2}}\sum_{r\in\mathcal{M}(N)}\sum_{a=M+1}^{2M}\Delta_I(a)^{\frac{1}{4s^2}}
\text{.}
\end{split}
\end{equation}
so we may finally turn our attention to estimating $\Delta_I(a)$ for $a\in[M+1,2M]$.

The estimate above holds with some constant depending only on $\alpha,P,\psi,\delta$ and uniform with respect to the choice of the interval $I$, and in the following step we will make an appropriate choice. More precisely, we will prove that for every frequency $N^{-d+\delta}\le |\xi|\le N$, there exists an $I=I(\xi,N)=\{j_0=j_0(\xi,N)\}$, $j_0\in[d,d_0]\cap\mathbb{N}$, such that a polynomial saving is achieved here.
\smallskip
\paragraph{\textbf{Step 5: Choosing $I$ and concluding.}} We begin by letting $I=\{j\}\subseteq[d,d_0]\cap\mathbb{N}$ be arbitrary and estimate $\Delta_I(a)$ for $a\in[M+1,2M]\subseteq [N/q,4N/q]$ using \eqref{derivativefull}  as follows
\begin{equation}
\label{DeltaIestbasic}
\begin{split}
\Delta_I(a)\coloneq&\big(10k\log(3sP)\big)^k\prod_{j\in I}\bigg(\bigg|\frac{F^{(j)}(a)}{j!}\bigg|+\bigg|\frac{ F^{(j)}(a)}{j!}\bigg|^{-1}X^{-2j}\bigg)
\\
\lesssim& (\log N)^{k}\bigg(|\xi|q^j|\theta|^{j-d}+|\xi|^{-1}q^{-j}|\theta|^{d-j}N^{\sigma}\Big(N/q\cdot N^{-\frac{\sigma}{d_0}}\Big)^{-j}\bigg)
\\
\lesssim& N^{\sigma}\Big(|\xi||\theta|^{j-d}N^{\rho j}+|\xi|^{-1}|\theta|^{d-j}|N^{\sigma}N^{-j} N^{\frac{\sigma}{d_0}j}\Big)
\\
\lesssim& N^{\sigma}\Big(|\xi||\theta|^{j-d}N^{\rho j}+|\xi|^{-1}|\theta|^{d-j}N^{-j+2\sigma}\Big)\text{.}
\end{split}
\end{equation}
Given  an $N\in\mathbb{N}$ and $\xi\in\mathbb{R}$ with $N^{-d+\delta}\le |\xi|\le N$, we let $j_0=j_0(\xi,N)\in[d, d_0]$ be the smallest integer such that 
\[
|\xi||\theta|^{j_0-d}N^{\rho j_0}\le N^{-2\sigma}\text{.}
\]
This is always possible since 
\begin{equation}\label{nonemptyj0}
\big\{j\in[d,d_0]\cap\mathbb{Z}:\,|\xi||\theta|^{j-d}N^{\rho j}\le N^{-2\sigma}\big\}\neq \emptyset\text{.}
\end{equation}
To see this, note that $d_0-1\ge d$ belongs in the set, since by \eqref{upperboundsfortheta} we may estimate as follows
\begin{equation}
\begin{split}
|\xi|&|\theta|^{(d_0-1)-d}N^{\rho (d_0-1)}\le NN^{\big(-\rho-\frac{\rho}{\omega'-1}\big)((d_0-1)-d)}N^{\rho (d_0-1)}
\\
=&NN^{-\rho (d_0-1)-\frac{\rho}{\omega'-1}(d_0-1)+\rho d+\frac{\rho}{\omega'-1}d}N^{\rho (d_0-1)}=N^{1-\frac{\rho}{\omega'-1}(d_0-1)+\rho d+\frac{\rho}{\omega'-1}d}\le N^{-1}\le N^{-2\sigma}\text{,}
\end{split}
\end{equation}
since $\sigma\le \frac{1-\rho}{4}\le 1/4$, and 
\begin{equation}
\begin{split}
&1-\frac{\rho}{\omega'-1}(d_0-1)+\rho d+\frac{\rho}{\omega'-1}d\le -1\iff 2+\frac{\rho}{\omega'-1}+\rho d+\frac{\rho}{\omega'-1}d\le \frac{\rho}{\omega'-1}d_0
\\
&\iff \frac{2(\omega'-1)}{\rho}+1+d(\omega'-1)+d\le d_0\iff\frac{2(\omega'-1)}{\rho}+1+d\omega'\le d_0\text{,}
\end{split}
\end{equation}
which is true, since by \eqref{s_0choice} we get that
\begin{equation}
d_0\ge d\omega'+4(\omega'-1)+1\ge  d\omega'+\frac{2(\omega'-1)}{\rho}+1\text{,} \quad\text{because $\rho>1/2$.}
\end{equation}
Thus, \eqref{nonemptyj0} is justified and we also get that $j_0\le d_0-1$. Since $j_0$ is the smallest with this property, we get that either $j_0-1=d-1$ or   
\begin{equation}
j_0-1\ge d\quad\text{and}\quad |\xi||\theta|^{j_0-1-d}N^{\rho (j_0-1)}> N^{-2\sigma}\text{.}
\end{equation}

In the first case, namely, if $j_0=d$, we get that
\[
|\xi|^{-1}|\theta|^{d-j_0}N^{-j_0+2\sigma}\le N^{d-\delta}N^{-d+2\sigma}=N^{-
\delta+2\sigma}\le N^{-2\sigma}\text{,}\quad\text{since $\sigma\le \frac{\delta}{4}$,}
\]
and, consequently, in this case we have
\[
\Delta_{\{j_0\}}(a)\lesssim N^{\sigma}\Big(|\xi||\theta|^{j_0-d}N^{\rho j_0}+|\xi|^{-1}|\theta|^{d-j_0}N^{-j_0+2\sigma}\Big)\lesssim N^{-\sigma}\text{.}
\]

In the second case, that is, if $j_0-1\ge d$, we must have that
\begin{equation}
\begin{split}
|\xi||\theta|^{j_0-1-d}N^{\rho (j_0-1)}> N^{-2\sigma}\Longrightarrow|\xi|^{-1}|\theta|^{d-j_0}< N^{2\sigma+\rho (j_0-1)}|\theta|^{-1}\text{,}
\end{split}
\end{equation}
and thus the last summand in the final bound of \eqref{DeltaIestbasic} can be bounded as follows for $I\coloneq \{j_0\}$
\begin{equation}
\begin{split}
|\xi|^{-1}|\theta|^{d-j_0}N^{-j_0+2\sigma}\le N^{2\sigma+\rho (j_0-1)}|\theta|^{-1}N^{-j_0+2\sigma}\lesssim N^{4\sigma+\rho j_0-\rho}N^{\rho\omega'}N^{-j_0}
\\
= N^{4\sigma-\rho+j_0(\rho-1)+\rho\omega'}\le N^{4\sigma-\rho+(d+1)(\rho-1)+\rho\omega'}
= N^{4\sigma-\rho+d\rho-d+\rho-1+\rho\omega'}\\
=N^{4\sigma+d\rho-d-1+\rho\omega'}=N^{4\sigma+\rho(d+\omega')-(d+1)}\le N^{-4\sigma}\text{,}
\end{split}
\end{equation}
since 
\[
\rho(d+\omega')-(d+1)\le -8\sigma
\]
because
\[
\sigma\le \frac{(d+1)-\rho(d+\omega')}{8}\text{,}\quad\text{by \eqref{tau'choice}.}
\]
But then, once again, we get that
\[
\Delta_{\{j_0\}}(a)\lesssim N^{\sigma}\Big(|\xi||\theta|^{j_0-d}N^{\rho j_0}+|\xi|^{-1}|\theta|^{d-j_0}N^{-j_0+2\sigma}\Big)\lesssim N^{-\sigma}\text{.}
\]
Thus, returning to \eqref{Step4end} and choosing $I=\{j_0=j_0(\xi,N)\}$, we get that
\begin{equation}\label{Step5end}
\begin{split}
\Big|\sum_{n=N+1}^{2N}e(\xi P(n)\psi(\alpha n))\Big|
\lesssim N^{\rho}+N^{1-\frac{\sigma}{d_0}}+N^{1-\sigma'}+N^{1-\frac{\rho}{\omega'-1}}+N^{\frac{\sigma}{200s^2}}\sum_{r\in\mathcal{M}(N)}\sum_{a=M+1}^{2M}\Delta_{\{j_0\}}(a)^{\frac{1}{4s^2}}&
\\
\lesssim N^{\rho}+N^{1-\frac{\sigma}{d_0}}+N^{1-\sigma'}+N^{1-\frac{\rho}{\omega'-1}}+N^{\frac{\sigma}{200s^2}}qMN^{-\frac{\sigma}{4s^2}}&
\\
\lesssim N^{\rho}+N^{1-\frac{\sigma}{d_0}}+N^{1-\sigma'}+N^{1-\frac{\rho}{\omega'-1}}+N^{1-\frac{\sigma}{8s^2}}&\text{.}
\end{split}
\end{equation}
Letting $\chi\coloneq \min(1-\rho,\frac{\sigma}{d_0},\sigma',\frac{\rho}{\omega'-1},\frac{\sigma}{8s^2})>0$, we have shown that for every $N\gtrsim 1$ and $\xi\in\mathbb{R}$ with $N^{-d+\delta}\le|\xi|\le N$ we have
\begin{equation}
\Big|\sum_{n=N+1}^{2N}e(\xi P(n)\psi(\alpha n))\Big|
\lesssim N^{1-\chi}\text{,}\quad\text{and the proof is complete.}
\end{equation}
\end{proof}
\subsection{Removing the floor} Proposition~$\ref{minarcprop}$ is the heart of the minor arc analysis and now all that remains is to establish the estimate for the floored phase. Although, the handling of the floor in the minor arc regime is not as straightforward as in the major arc one, see \eqref{removefloor}, its complexity is nowhere near the one of the previous proof. Our arguments are quite generic, in the sense that the specific form of the sequences $(P(n)\psi(\alpha n))_{n\in\mathbb{N}}$ is not utilized. The first step is to establish a full (or non-dyadic) variant of Proposition~$\ref{minarcprop}$. This is standard, but for the convenience of the reader we provide the details.
\begin{proposition}\label{minfull}
Assume that $\psi\colon \mathbb{R}\to\mathbb{R}$ is a smooth $1$-periodic function such that none of its derivatives has a root of infinite order. Let $\alpha\in\mathbb{R}$ be such that  $\omega(\alpha)\eqcolon \omega\in[2,\infty)$, let $P\in\mathbb{R}[x]$ be of degree $d\ge 1$ and assume that
\begin{equation}\label{mudrestrsame}
\omega<1+\sqrt{d+1}\text{.}
\end{equation}
Then for every $\delta'>0$ there exist positive constants $\chi'=\chi'(\alpha,P,\psi,\delta')$ and $C=C(\alpha,P,\psi,\delta')$ such that for every $N\in\mathbb{N}$ and $\xi\in\mathbb{R}$ with $N^{-d+\delta'}\le |\xi|\le N^{1/2}$ we have
\begin{equation}\label{goalfullexpsummin}
\bigg|\sum_{n=1}^Ne\big(\xi P(n)\psi(\alpha n)\big)\bigg|\le CN^{1-\chi'}\text{.}
\end{equation}
\end{proposition}
\begin{proof}Fix $\psi,\alpha,P,\delta'$ as in the statement of the proposition, and without loss of generality assume that $\delta'\in(0,1)$. All implicit constants may depend on these fixed parameters. We apply Proposition~$\ref{minarcprop}$ with $\delta=\delta'/2$, to deduce the existence of a positive real $\chi=\chi(\alpha,P,\psi,\delta)$ such that for every $t\ge1$ and $\zeta\in\mathbb{R}$ with $t^{-d+\delta'/2}\le |\zeta|\le t$ we have
\begin{equation}\label{expsummingen}
\Big|\sum_{t<n\le 2t}e(\zeta P(n)\psi\big(\alpha n)\big)\Big|\lesssim t^{1-\chi}\text{,}
\end{equation}
and let us assume without loss of generality that $\chi< \min(1/2,\delta/4)$.

It suffices to establish \eqref{goalfullexpsummin}  for $N^{-d+\delta'}<\xi\le N^{1/2}$ and $N\gtrsim 1$.  Let $K\coloneqq \big\lfloor \frac{\delta'}{2d}(\log_2 N)\big\rfloor$, and estimate as follows
\begin{equation}\label{splitcalc}
\begin{split}
\bigg|\sum_{n\le N}e\big(\xi P(n)\psi(\alpha n)\big)\bigg|\le& \sum_{i=0}^{K-1}\bigg|\sum_{N/2^{i+1}<n\le N/2^i}e\big(\xi P(n)\psi(\alpha n)\big)\bigg|+\bigg|\sum_{1\le n\le N/2^K
}e\big(\xi P(n)\psi(\alpha n)\big)\bigg|
\\
\le &\sum_{i=0}^{K-1}\bigg|\sum_{N/2^{i+1}<n\le N/2^i}e\big(\xi P(n)\psi(\alpha n)\big)\bigg|+N2^{-K}\text{.}
\end{split}
\end{equation}
Note that
\[
N2^{-K}\le 2N2^{-K-1}\le 2N2^{-\frac{\delta'}{2d}\log_2N}=2N^{1-\frac{\delta'}{2d}}\text{.}
\]
We wish to use the estimate \eqref{expsummingen} for every dyadic piece above with $\delta=\delta'/2$.  We claim that for $N\gtrsim 1$ and every $i\in\{0,\dotsc,K-1\}$ we have 
\begin{equation}\label{toapplyexp}
(N/2^{i+1})^{-d+\delta'/2}\le\xi\le N/2^{i+1}\text{.}
\end{equation}

For the second inequality, we note that since $\xi<N^{1/2}$ it suffices to check that $N^{1/2}\le N/2^{i+1}$, which does hold since
\begin{equation}
\begin{split}
N^{1/2}\le N/2^{i+1}\iff 2^{i+1}\le N^{1/2}\iff i+1\le \frac{1}{2}\log_2N\text{,}
\end{split}
\end{equation}
which is true because $i+1\le K\le \frac{\delta'}{2d}\log_2 N\le \frac{1}{2}\log_2 N$. 

For the lower bound, we note that since $\xi\ge N^{-d+\delta'}$ it suffices to prove that $N^{-d+\delta'}\ge(N/2^{i+1})^{-d+\delta'/2} $. To see this, note that
\begin{equation}
\begin{split}
N^{-d+\delta'}\ge(N/2^{i+1})^{-d+\delta'/2} \iff& N^{-d+\delta'}\ge N^{-d+\delta'/2}2^{(i+1)(d-\delta'/2)} \iff N^{\delta'/2}\ge 2^{(i+1)(d-\delta'/2)} 
\\
\iff& \frac{\delta'}{2}\log_2 N\ge (i+1)(d-\delta'/2)\iff i+1\le \frac{\delta'}{2d-\delta'}\log_2 N\text{,}
\end{split}
\end{equation}
which is true, since
\[
i+1\le K\le\frac{\delta'}{2d}\log_2 N\le  \frac{\delta'}{2d-\delta'}\log_2 N\text{.}
\]
Thus, for every $i\in\{0,\dotsc,K-1\}$ the estimate \eqref{expsummingen} is applicable for every $N^{-d+\delta'}<\xi\le N^{1/2}$ for $N\gtrsim1$ and yields
\begin{equation*}
\begin{split}
\bigg|\sum_{n\le N}e\big(\xi P(n)\psi(\alpha n)\big)\bigg|
\lesssim& \sum_{i=0}^{K-1}\bigg|\sum_{N/2^{i+1}<n\le N/2^i}e\big(\xi P(n)\psi(\alpha n)\big)\bigg|+N^{1-\frac{\delta'}{2d}}
\\
\lesssim &\sum_{i=0}^{K-1}(N/2^i)^{1-\chi}+N^{1-\frac{\delta'}{2d}}=N^{1-\chi}\sum_{i=0}^{K-1}2^{i(\chi-1)}+N^{1-\frac{\delta'}{2d}}
\\
\le& N^{1-\chi}\sum_{i=0}^{\infty}\Big(1/\sqrt{2}\Big)^{i}+ N^{1-\frac{\delta'}{2d}}\lesssim N^{1-\chi}+N^{1-\frac{\delta'}{2d}}\text{,}
\end{split}
\end{equation*}
and taking $\chi'\coloneq\min\big(\chi,\frac{\delta'}{2d}\big)>0$ completes the proof.
\end{proof}
We now state a general auxiliary floor-removing lemma.
\begin{lemma}[Floor-removing lemma]\label{floorremovgeneric} Fix a sequence of real numbers $(a_n)_{n\in\mathbb{N}}$ and assume that there exist positive constants $l_1,l_2,C',\chi'$ such that for every $N\in\mathbb{N}$ and $\xi\in\mathbb{R}$ with $N^{-l_1}\le|\xi|\le N^{l_2}$ we have
\begin{equation}\label{unflooredgeneric}
\Big|\sum_{n=1}^Ne(\xi a_n)\Big|\le C'N^{1-\chi'}\text{.}
\end{equation}
Then there exist positive constants $C,\chi$ such that for every $N\in\mathbb{N}$ and $\xi\in\mathbb{R}$ with $\|\xi\|\ge N^{-l_1}$ we have
\begin{equation}\label{goalfloor}
\Big|\sum_{n=1}^Ne(\xi\lfloor a_n\rfloor)\Big|\le CN^{1-\chi}\text{.}
\end{equation} 
\end{lemma}
\begin{proof}
Fix a sequence of real numbers $(a_n)_{n\in\mathbb{N}}$ and let $l_1,l_2,C',\chi'$ be positive real numbers as in the assumption of the lemma. Without loss of generality we may assume that $\chi'<l_2/4$ since if the condition holds for some $\chi'$, it holds for every smaller positive number. It suffices to establish the estimate $\eqref{goalfloor}$ for $N\gtrsim1$ and $\xi\in\mathbb{R}$ with $N^{-l_1}\le |\xi|\le 1/2$. We will use a standard truncated Fourier series to remove the floor. For every $M\ge 2$, $x\in[-1/2,1/2)\setminus\{0\}$, and $y\in\mathbb{R}$, we have that
\[
e(-x\{y\})=\sum_{|m|\le M}\frac{1-e(-x)}{2\pi i (m+x)}e(my)+O\bigg(\min\bigg(1,\frac{1}{M\|y\|}\bigg)\bigg)\text{,}
\]
where the implied constant is absolute, see page 29 in \cite{LPE}. We apply this for $M=\big\lfloor N^{l_2-\frac{\chi'}{2}}-\frac{1}{2}\big\rfloor$ and obtain
\begin{equation}\label{1stapproxgeneric}
\sum_{n\le N}e(\xi\lfloor a_n\rfloor)=\sum_{n\le N}e(\xi a_n)\sum_{|m|\le M}\frac{1-e(-\xi)}{2\pi i (m+\xi)}e(ma_n)+\sum_{n\le N}e(\xi a_n)g_M(\xi,n)\text{,}
\end{equation}
where $|g_M(\xi,n)|\lesssim \min\Big(1,\frac{1}{M\|a_n\|}\Big)$.

We begin by bounding the second term. Note that 
\begin{equation}\label{ESTerrmin}
\Big|\sum_{n\le N}e(\xi a_n)g_M(\xi,n)\Big|\lesssim \sum_{n\le N} \min\bigg(1,\frac{1}{M\|a_n\|}\bigg)\text{.}
\end{equation}
We now use another standard Fourier series expansion, namely, for every $M\in\mathbb{N}$ and $x\in\mathbb{R}$ we have that
\begin{equation}\label{minexp}
\min\bigg(1,\frac{1}{M\|x\|}\bigg)=\sum_{m\in\mathbb{Z}}b_me(mx)\text{,}\quad\text{where $|b_m|\lesssim \min\bigg(\frac{1+\log M}{M},\frac{M}{|m|^2}\bigg)$, see page 23 in \cite{LPE}. }
\end{equation}
The coefficients $b_m$ also depend on $M$, but we chose to suppress this dependence. We have that
\begin{equation}\label{easyboundfornoncanc}
\sum_{1\le n\le N}\min\bigg(1,\frac{1}{M\|a_n\|}\bigg)=\sum_{1\leq n\le N}\sum_{m\in\mathbb{Z}}b_me(ma_n)\le \sum_{m\in\mathbb{Z}}|b_m|\Big|\sum_{1\leq n\leq N}e(ma_n)\Big|\text{.}
\end{equation}
For $N\gtrsim1$ we may estimate using \eqref{unflooredgeneric} and the bounds for $|b_m|$ in \eqref{minexp} as follows
\begin{equation}\label{easyboundfinal}
\begin{split}
\sum_{1\le n\le N}\min\bigg(1,\frac{1}{M\|a_n\|}\bigg)\le&
\sum_{m\in\mathbb{Z}}|b_m|\Big|\sum_{1\le n\le N}e(ma_n)\Big|\\
=&|b_0|N+\sum_{1\le |m|\le N^{l_2}}|b_m|\Big|\sum_{1\le n\le N}e(ma_n)\Big|+\sum_{|m|> N^{l_2}}|b_m|N
\\
\lesssim&\frac{N\log M}{M}+\sum_{1\le|m|\le N^{l_2}}\frac{\log M}{M}N^{1-\chi'}+\sum_{|m|>N^{l_2}}\frac{M}{|m|^2}N
\\
\lesssim& (\log N)NM^{-1}+(\log N)N^{1-\chi'}M^{-1}N^{l_2}+N MN^{-l_2}\lesssim
\\
\lesssim& (\log N)N^{1-l_2+\frac{\chi'}{2}}+(\log N)N^{1-\chi'}N^{\frac{\chi'}{2}}+N^{1-\frac{\chi'}{2}}\lesssim N^{1-\frac{\chi'}{3}}
\text{.}
\end{split}
\end{equation}
Combining \eqref{ESTerrmin} with \eqref{easyboundfinal} yields that
\begin{equation}\label{S2boundeasy}
\Big|\sum_{n\le N}e(a_n\xi)g_M(\xi,n)\Big|\lesssim N^{1-\frac{\chi'}{3}}\text{,}
\end{equation}
and concludes the estimation of the second summand in $\eqref{1stapproxgeneric}$. 

For the first summand of  \eqref{1stapproxgeneric} we may estimate using once again \eqref{unflooredgeneric} as follows for $N\gtrsim1 $
\begin{equation}\label{S1boundeasy}
\begin{split}
\Big|\sum_{n\le N}e(\xi a_n)&\sum_{|m|\le M}\frac{1-e(-\xi)}{2\pi i (m+\xi)}e(m a_n)\Big|\lesssim \sum_{|m|\le M}\frac{|1-e(-\xi)|}{|m+\xi|}\Big|\sum_{n\le N}e((\xi+m)a_n)\Big|
\\
\lesssim&\Big|\sum_{n\le N}e(\xi a_n)\Big|+\sum_{1\le |m|\le M}\frac{1}{|m|}\Big|\sum_{n\le N}e((\xi+m)a_n)\Big|
\\
\lesssim&\log M\sup_{|m|\le M}\Big|\sum_{n\le N}e((\xi+m)h(n))\Big|
\\
\lesssim& \log N\sup_{N^{-l_1}\le |t|\le M+1/2}\Big|\sum_{n\le N}e(th(n))\Big|
\\
\le& \log N\sup_{N^{-l_1}\le |t|\le N^{l_2}}\Big|\sum_{n\le N}e(th(n))\Big|\lesssim N^{1-\frac{\chi'}{3}}\text{.}
\end{split}
\end{equation}
Choosing $\chi\coloneq\frac{\chi'}{3}$, and taking into account \eqref{1stapproxgeneric}, \eqref{S2boundeasy} and \eqref{S1boundeasy} completes the proof.
\end{proof}
Now the proof of Proposition~$\ref{minarcpropfull}$ is immediate. For the sake of completeness we give a short proof here.

\begin{proof}[Proof of Proposition~$\ref{minarcpropfull}$]
Fix $\psi,\alpha,P,\delta$ as in the statement of the proposition, and without loss of generality assume that $\delta\in(0,1)$. Then Proposition~$\ref{minfull}$ is applicable and yields that there exist positive constants $\chi'=\chi'(\alpha,P,\psi,\delta)$ and $C=C(\alpha,P,\psi,\delta)$ such that for every $N\in\mathbb{N}$ and $\xi\in\mathbb{R}$ with $N^{-d+\delta}\le |\xi|\le N^{1/2}$ we have
\begin{equation*}
\bigg|\sum_{n=1}^Ne\big(\xi P(n)\psi(\alpha n)\big)\bigg|\le CN^{1-\chi'}\text{.}
\end{equation*}
Using this exponential sum estimate and appealing to Lemma~$\ref{floorremovgeneric}$  for the sequence $(P(n)\psi(\alpha n))_{n\in\mathbb{N}}$ with $l_1=d-\delta>0$ and $l_2=1/2$ immediately yields the conclusion of  Proposition~$\ref{minarcpropfull}$.
\end{proof}
\subsection{Minor arc estimates for the approximant.} Finally, let us give a quick proof of the much easier minor arc estimates for the continuous counterpart of the operator, namely, establish the estimate \eqref{minarcestcompleteeasy}. \begin{proposition}\label{psireg}
Let $\psi\colon[0,1]\to\mathbb{R}$ be a continuous function satisfying the following: for every $\varepsilon>0$ there exist positive constants $C_0=C_0(\varepsilon),\nu=\nu(\varepsilon)$ such that for every $t\ge 1$ we have
\begin{equation}\label{psiconditionzeroregion}
\lambda(\{s\in[0,1]:\,|\psi(s)|<t^{-\varepsilon}\})\le C_0t^{-\nu}\text{.}
\end{equation}
Let $P\in\mathbb{R}[x]$ be a polynomial of degree $d\ge 1$. Then for every $\delta>0$ there exist positive constants $C_1=C_1(P,\delta,\psi)$ and $\rho=\rho(d,\delta,\psi)$ such that for every $t\ge 1$ and $|\xi|\ge t^{-d+\delta}$ we have
\begin{equation}\label{minorarccont}
\bigg|\int_0^1\bigg(\frac{1}{t}\int_0^te(\xi P(u)\psi( s))du\bigg)ds\bigg|\le C_1t^{-\rho}\text{.}
\end{equation}
\end{proposition}
\begin{proof}
Let $P\in\mathbb{R}[x]$ be of degree $d\in\mathbb{N}$ and let $\delta>0$. Then for every $t\ge 1$ we perform the following decomposition
\begin{equation}\label{splitforvdCnew}
\begin{split}
&\int_0^1\bigg(\frac{1}{t}\int_0^te(\xi P(u)\psi( s))du\bigg)ds
\\
&=\int_{\substack{s\in[0,1]:\\|\psi(s)|\ge t^{-\delta/2}}}\bigg(\frac{1}{t}\int_0^te(\xi P(u)\psi( s))du\bigg)ds+\int_{\substack{s\in[0,1]:\\|\psi(s)|< t^{-\delta/2}}}\bigg(\frac{1}{t}\int_0^te(\xi P(u)\psi( s))du\bigg)ds\text{.}
\end{split}
\end{equation} 
We estimate the second summand using the trivial estimate for the integrand and the assumption for $\psi$ with $\delta/2>0$, so as to obtain two positive real numbers $C_0=C_0(\delta/2)$ and $\nu=\nu(\delta/2)$ such that
\begin{equation}\label{trivialestnew}
\bigg|\int_{\substack{s\in[0,1]:\\|\psi(s)|< t^{-\delta/2}}}\Big(t^{-1}\int_0^te(\xi P(u)\psi( s))du\Big)ds\bigg|\le \lambda\Big(\Big\{s\in[0,1]:\,|\psi(s)|<t^{-\frac{\delta}{2}}\Big\}\Big) \le C_0t^{-\nu}\text{.}
\end{equation}
For the first, we use van der Corput since the $d$-th derivative of the phase can be easily calculated 
\[
\big(\xi P(\cdot)\psi( s)\big)^{(d)}=\xi\psi(s)C_Pd!\text{,}\quad\text{where $C_P\neq 0$ is the leading coefficient of $P$,}
\]
and thus for every $\xi\in\mathbb{R}$ with $|\xi|\ge t^{-d+\delta}$ and $s\in[0,1]$ with $|\psi(s)|\ge t^{-\delta/2}$ we have
\[
\Big|\big(\xi P(\cdot)\psi( s)\big)^{(d)}\Big|\ge t^{-d+\delta/2}d!|C_P|\text{.}
\]
Van der Corput's lemma, see Proposition~2.6.7(b) in \cite{Grafakos}, yields that for every such $\xi$ and $s$ we have 
\begin{equation}\label{vdCoutputnew}
\bigg|\frac{1}{t}\int_0^te(\xi P(u)\psi( s))du\bigg|\le t^{-1}12d\big(t^{-d+\delta/2}d!|C_P|\big)^{-1/d}=12d\big(d!|C_P|\big)^{-1/d}\cdot t^{-\frac{\delta}{2d}}\text{.}
\end{equation}
By \eqref{splitforvdCnew}, \eqref{trivialestnew} and \eqref{vdCoutputnew}, we obtain the desired estimate for $\rho=\rho(d,\delta,\psi)\coloneq \min\big(\frac{\delta}{2d},\nu\big(\frac{\delta}{2}\big)\big)>0$ since
\begin{equation*}
\Big|\int_0^1\Big(t^{-1}\int_0^te(\xi P(u)\psi( s))du\Big)ds\Big|\lesssim_{P,\delta,\psi}t^{-\frac{\delta}{2d}}+t^{-\nu} \lesssim_{P,\delta,\psi}t^{-\rho}\text{,}\quad\text{and the proof of \eqref{minorarccont} is complete.}
\end{equation*}
\end{proof}
\subsection{Concluding the one-frequency circle method.}
We end the present section by combining the major and minor arc analysis estimates and giving a quick proof of the approximation result of Proposition~$\ref{fam}$.
\begin{proof}[Proof of Proposition~$\ref{fam}$]
By choosing $\delta\coloneq \frac{1}{2(2\omega-1)}>0$ both Proposition~$\ref{Majarcprop}$ and Proposition~$\ref{minarcpropfull}$ are applicable, and yield the estimates \eqref{majarcestcomplete} and \eqref{minarcestcompletedif}, respectively. For the easier estimate \eqref{minarcestcompleteeasy}, we note that since, by assumption, $\psi$ is a smooth $1$-periodic function such that none of its derivatives has a root of infinite order, we get, in particular, that $\psi'$ does not have a root of infinite order. Lemma~$\ref{elementarybutdifficult}$ is applicable, and shows that $\psi$ satisfies the condition $\ref{psiconditionzeroregion}$. In turn, Proposition~$\ref{psireg}$ may be applied and immediately yields \eqref{minarcestcompleteeasy} for our choice of $\delta$. Clearly, the existence of a positive $\delta=\delta(\omega)$ making \eqref{majarcestcomplete}, \eqref{minarcestcompletedif} and \eqref{minarcestcompleteeasy} hold, immediately yields
\eqref{fullestcomplete}, completing the proof of Proposition~$\ref{fam}$. 
\end{proof} 
\section{Variational estimates for the continuous counterpart}\label{contandconclude}
In this short section we give a simple proof of the $L^p(\mathbb{R})$-boundedness of the $r$-variations for the continuous counterpart of the approximant operator. More precisely, we prove the following. 
\begin{proposition}\label{contvarprop}
Let $\psi\colon[0,1]\to \mathbb{R}$ be a continuous\footnote{We opted for a formulation of the result requiring $\psi$ to be continuous to avoid issues of measurability.} function and let $P\in\mathbb{R}[x]$ be a polynomial of degree $d\ge 1$. For every $N\in\mathbb{N}$, $s\in[0,1]$, $f\in L^1_{\normalfont{\text{loc}}}(\mathbb{R})$ and $x\in\mathbb{R}$, let
\begin{equation}\label{Atdef}
B_{N,s}f(x)\coloneq\frac{1}{N}\int_{0}^Nf\big(x-\psi(s)P(u)\big)du\quad\text{and}\quad B_{N}f(x)\coloneq\int_0^1B_{N,s}f(x)ds\text{.}
\end{equation}
Then for every $p\in(1,\infty)$ and $r\in(2,\infty)$ there exists a positive constant $C=C(p,d)$ such that for every $f\in L^p(\mathbb{R})$ we have
\begin{equation}
\big\|V^r(B_{N}f:\,N\in\mathbb{N})\big\|_{L^p(\mathbb{R})}\le C\frac{r}{r-2}\|f\|_{L^p(\mathbb{R})}\text{.}
\end{equation}
\end{proposition}
\begin{proof}
The proof relies on the fact that for every polynomial $Q\in \mathbb{R}[x]$ of degree $d\in\mathbb{N}$ with $Q(0)=0$, every $p\in(1,\infty)$ and every $r\in(2,\infty)$ there exists a constant $C(p,d)$ such that
\begin{equation}\label{marvarinput}
\big\|V^r\big(\mathcal{M}^Q_tf:\,t>0\big)\big\|_{L^p(\mathbb{R})}\le C(p,d)\frac{r}{r-2}\|f\|_{L^p(\mathbb{R})}\text{,}\quad\text{where}\quad
\mathcal{M}^Q_tf(x)=\frac{1}{t}\int_0^tf(x-Q(u))du\text{,}
\end{equation}
see Theorem~A.1. in \cite{marvar}. The crucial part here is that the constant does not depend on the coefficients. Firstly, we note that for every $r>2$ and $x\in\mathbb{R}$ we have
\begin{equation}
\begin{split}
V^r(B_{N}f(x):\,N\in\mathbb{N})=&\sup_{J\in\mathbb{N}}\sup_{t_0<\cdots<t_J}\Big(\sum_{j=1}^J|B_{t_j}f(x)-B_{t_{j-1}}f(x)|^r\Big)^{1/r}
\\
=&\sup_{J\in\mathbb{N}}\sup_{t_0<\cdots<t_J}\Big(\sum_{j=1}^J\Big|\int_0^1\big(B_{t_j,s}f(x)-B_{t_{j-1},s}f(x)\big)ds\Big|^r\Big)^{1/r}
\\
\le&\sup_{J\in\mathbb{N}}\sup_{t_0<\cdots<t_J}\int_0^1\Big(\sum_{j=1}^J\big|B_{t_j,s}f(x)-B_{t_{j-1},s}f(x)\big|^r\Big)^{1/r}ds
\\
\le&\int_0^1\sup_{J\in\mathbb{N}}\sup_{t_0<\cdots<t_J}\Big(\sum_{j=1}^J\big|B_{t_j,s}f(x)-B_{t_{j-1},s}f(x)\big|^r\Big)^{1/r}ds
\\
=&\int_{\substack{s\in[0,1]:\\\psi(s)\neq0}}\sup_{J\in\mathbb{N}}\sup_{t_0<\cdots<t_J}\Big(\sum_{j=1}^J\big|B_{t_j,s}f(x)-B_{t_{j-1},s}f(x)\big|^r\Big)^{1/r}ds\text{,}
\end{split}
\end{equation} 
since for $s\in[0,1]$ such that $\psi(s)=0$, we have that $B_{t,s}f(x)-B_{t',s}f(x)=0$ for every $t,t'\in\mathbb{N}$. Finally, we note that for every $p\in(1,\infty)$ we have
\begin{equation*}
\begin{split}
\big\|V^r\big(B_{N}f:\,N\in\mathbb{N}\big)\big\|_{L^p(\mathbb{R})}\le \bigg(\int_{\mathbb{R}}\bigg|\int_{\substack{s\in[0,1]:\\\psi(s)\neq0}}\sup_{J\in\mathbb{N}}\sup_{t_0<\cdots<t_J}\Big(\sum_{j=1}^J\big|B_{t_j,s}f(x)-B_{t_{j-1},s}f(x)\big|^r\Big)^{1/r}ds\bigg|^pdx\bigg)^{1/p}&
\\
\le\int_{\substack{s\in[0,1]:\\\psi(s)\neq0}}\bigg(\int_{\mathbb{R}}\bigg|\sup_{J\in\mathbb{N}}\sup_{t_0<\cdots<t_J}\Big(\sum_{j=1}^J\big|B_{t_j,s}f(x)-B_{t_{j-1},s}f(x)\big|^r\Big)^{1/r}dx\bigg|^p\bigg)^{1/p}ds &
\\
=\int_{\substack{s\in[0,1]:\\\psi(s)\neq0}}\bigg(\int_{\mathbb{R}}\bigg|\sup_{J\in\mathbb{N}}\sup_{t_0<\cdots<t_J}\Big(\sum_{j=1}^J\big|B_{t_j,s}f(x+\psi(s)P(0))-B_{t_{j-1},s}f(x+\psi(s)P(0))\big|^r\Big)^{1/r}dx\bigg|^p\bigg)^{1/p}ds&
\\
=\int_{\substack{s\in[0,1]:\\\psi(s)\neq0}}\big\|V^r\big(\mathcal{M}^{\psi(s)(P(\cdot)-P(0))}_tf:\,t\in\mathbb{N}\big)\big\|_{L^p(\mathbb{R})}ds &
\\
\le\int_{\substack{s\in[0,1]:\\\psi(s)\neq0}}C(p,d)\frac{r}{r-2}\|f\|_{L^p(\mathbb{R})}ds\le C(p,d)\frac{r}{r-2}\|f\|_{L^p(\mathbb{R})}&\text{,}
\end{split}
\end{equation*}
where we have used \eqref{marvarinput}, the fact that for every $s\in[0,1]$, $t\in\mathbb{N}$, and $x\in\mathbb{R}$ we have
\begin{equation}
\begin{split}
B_{t,s}f(x+\psi(s)P(0))=&\frac{1}{t}\int_{0}^tf\big(x+\psi(s)P(0)-\psi(s)P(u)\big)du
\\
=&\frac{1}{t}\int_{0}^tf\big(x-\psi(s)(P(u)-P(0))\big)du=\mathcal{M}^{\psi(s)(P(\cdot)-P(0))}_tf(x)\text{,}
\end{split}
\end{equation}
and the fact that $\psi(s)(P(\cdot)-P(0))$ is a polynomial with no constant coefficient and of degree $d$ for every $s\in[0,1]$ such that $\psi(s)\neq0$.
\end{proof}
\section{Proof of Theorem~$\ref{qMain}$}
The present and final section is devoted to the proof of Theorem~$\ref{qMain}$. The main tools are, on the one hand, the output of the one-frequency circle method, see Proposition~$\ref{fam}$, and, on the other hand, the continuous $r$-variational estimates for the approximant operator, see Proposition~$\ref{contvarprop}$. These two ingredients yield the corresponding quantitative pointwise ergodic theorem of Theorem~$\ref{qMain}$, through delicate but, nevertheless, standard arguments. The process of utilizing such inputs to deduce variational estimates (in the context of lacunary times) has been carefully carried out in \cite{frld}, see Section~4 therein. For the sake of completeness we give a quick proof-sketch here, essentially following the aforementioned section.   

We use the following lemma to treat the ``so-called'' short variations.
\begin{lemma}\label{lemmashort}
Let $p\in(1,\infty)$ and assume $\big(T_n\colon\ell^p(\mathbb{Z})\to\ell^p(\mathbb{Z})\big)_{n\in\mathbb{N}}$ is a family of operators satisfying
\[
\sup_{n\in\mathbb{N}}\frac{n}{\log (2n)}\|T_{n+1}-T_{n}\|_{\ell^p(\mathbb{Z})\to\ell^p(\mathbb{Z})}\eqcolon M<\infty\text{.}
\]
Then there exist positive constants $C=C(M,p)$ and $\tau=\tau(p)\in(0,1)$ such that
\begin{equation}\label{Shortgen}
\bigg\|\bigg(\sum_{k=0}^{\infty}V_2\Big(T_nf:\,n\in\big[\big\lfloor 2^{k^{\tau}}\big\rfloor,\big\lfloor 2^{(k+1)^{\tau}}\big\rfloor\big)\Big)^2\bigg)^{1/2}\bigg\|_{\ell^p(\mathbb{Z})}\le C\|f\|_{\ell^{p}(\mathbb{Z})}\text{.}
 \end{equation}
\end{lemma}
\begin{proof}
We refer the reader to Proposition~4.1 in \cite{UA}.
\end{proof}
\begin{proof}[Proof of Theorem~$\ref{qMain}$]The variational estimate readily implies the $\mu$-a.e. convergence of the ergodic averages, and by the Calder\'on transference principle it suffices to establish it for the integer shift system, namely, to establish the following estimate for $p\in(1,\infty)$ and $r\in(2,\infty)$
\begin{equation}\label{vrint}
\big\|V^r\big(A_Nf:\,N\in\mathbb{N}\big)\big\|_{\ell^p(\mathbb{Z})}\lesssim_{p,\alpha,\psi,P}\frac{r}{r-2}\|f\|_{\ell^p(\mathbb{Z})}\text{,}
\end{equation}
where
\[
A_Nf(x)\coloneq \frac{1}{N}\sum_{n=1}^Nf\big(x-\lfloor P(n)\psi(\alpha n)\rfloor\big)\text{.}
\]
Let us fix $\alpha,\psi,P$ as in the hypotheses of the theorem and $p\in(1,\infty)$. All implicit constants may depend on these fixed choices without further mention.
\smallskip
\paragraph{\textbf{Step 1: Passing to sublacunary sequences}} It is not difficult to check that
\begin{equation}
\label{eq:17}
\big\|A_{N+1}f-A_{N}f\big\|_{\ell^p(\mathbb{Z})}
\lesssim N^{-1}\|f\|_{\ell^p(\mathbb{Z})}\text{,}
\end{equation}
and letting $\tau=\tau(p)\in(0,1)$ be the real number Lemma~$\ref{lemmashort}$ affords us, we begin by splitting the variations into long and short as follows for every $r\in(2,\infty)$
\begin{equation}\label{splitls}
\begin{split}
\big\|V^r\big(A_{N}f:\,N\in\mathbb{N}\big)\big\|_{\ell^p(\mathbb{Z})}\lesssim
&\big\|V^r\big(A_{\lfloor 2^{k^{\tau}}\rfloor}f:\,k\in\mathbb{N}_0\big)\big\|_{\ell^p(\mathbb{Z})}
\\
+&\bigg\|\bigg(\sum_{k=0}^{\infty}V^2\Big(A_{N}f:\,N\in\big[\big\lfloor 2^{k^{\tau}}\big\rfloor,\big\lfloor 2^{(k+1)^{\tau}}\big\rfloor\big)\cap\mathbb{N}\Big)^2\bigg)^{1/2}\bigg\|_{\ell^p(\mathbb{Z})}\text{.}
\end{split}
\end{equation}
We can now immediately bound the second summand by the aforementioned lemma, and obtain that
\begin{equation}\label{finalsplitlongonly}
\big\|V^r\big(A_{N}f:\,N\in\mathbb{N}\big)\big\|_{\ell^p(\mathbb{Z})}\lesssim
\big\|V^r\big(A_{\lfloor 2^{k^{\tau}}\rfloor}f:\,k\in\mathbb{N}_0\big)\big\|_{\ell^p(\mathbb{Z})}
+\|f\|_{\ell^p(\mathbb{Z})}\text{.}
\end{equation}
It suffices to treat the long oscillations, namely, to bound the first summand above.
\smallskip
\paragraph{\textbf{Step 2: Transferring the continuous variational estimate}}
For every $N\in\mathbb{N}$ and $\xi\in\mathbb{R}$ we let for convenience
\[
\mathcal{L}_N(\xi)\coloneq\int_0^1\bigg(\frac{1}{N}\int_0^Ne(\xi\psi(s)P(u))du\bigg)ds\text{,}
\]
and it is not difficult to see that
\begin{equation}
B_Nf\coloneq \mathcal{F}^{-1}_{\mathbb{R}}\big[\mathcal{L}_N\mathcal{F}_{\mathbb{R}}[f]\big]\quad\text{for every $f\in L^p(\mathbb{R})\cap L^2(\mathbb{R})$, where $B_Nf$ is defined in \eqref{Atdef}.}
\end{equation}
By Proposition~$\ref{contvarprop}$ we get that for every $f\in L^p(\mathbb{R})\cap L^2(\mathbb{R})$ we have
\begin{equation}\label{contmultvar}
\big\|V^r\big(\mathcal{F}^{-1}_{\mathbb{R}}\big[\mathcal{L}_N\mathcal{F}_{\mathbb{R}}[f]\big]:\,N\in\mathbb{N}\big)\big\|_{L^p(\mathbb{R})}=\big\|V^r(B_Nf:\,N\in\mathbb{N})\big\|_{L^p(\mathbb{R})}\lesssim \frac{r}{r-2}\|f\|_{L^p(\mathbb{R})}\text{.}
\end{equation}
We wish to apply Proposition~3.1 from \cite{marvar} so as to use this estimate to obtain one over $\mathbb{Z}$. To this end, we fix $\eta\in\mathcal{C}^{\infty}_c(\mathbb{R})$ be such that $1_{[-1/32,1/32]}\le \eta\le 1_{[-1/16,1/16]}$ and such that it is a convolution of two nonnegative smooth functions with compact supports contained in $(-1/16,1/16)$. By Proposition~3.1 in \cite{marvar} (applied for $Q=1$, $m=0$), the estimate \eqref{contmultvar} may be transferred to $\ell^p(\mathbb{Z})$, and we obtain
\begin{equation}\label{rvarforZforapproximant}
\begin{split}
\big\|V^r\big(\mathcal{F}^{-1}_{\mathbb{Z}}\big[\mathcal{L}_{N}\eta\mathcal{F}_{\mathbb{Z}}[f]\big]:\,N\in\mathbb{N}\big)\big\|_{\ell^p(\mathbb{Z})}
\lesssim\frac{r}{r-2} \big\|\mathcal{F}^{-1}_{\mathbb{Z}}\big[\eta\mathcal{F}_{\mathbb{Z}}[f]\big]\big\|_{\ell^p(\mathbb{Z})}
=\frac{r}{r-2} \|\mathcal{F}^{-1}_{\mathbb{Z}}[\eta]*f\|_{\ell^p(\mathbb{Z})}&
\\
\le\frac{r}{r-2} \|\mathcal{F}^{-1}_{\mathbb{Z}}[\eta]\|_{\ell^1(\mathbb{Z})}\|f\|_{\ell^p(\mathbb{Z})}\lesssim\frac{r}{r-2}\|f\|_{\ell^p(\mathbb{Z})}
&\text{,}
\end{split}
\end{equation}
where we have identified $\mathbb{T}$ with $[-1/2,1/2)$, and we have used Young's convolution inequality and the fact that $\eta$ is smooth so as to obtain $|\mathcal{F}^{-1}_{\mathbb{Z}}[\eta](n)|\lesssim_{\eta}n^{-2}$, see Theorem~3.3.9 in \cite{Grafakos}, for example.
\smallskip
\paragraph{\textbf{Step 3: Using our multiplier approximation to conclude.}} For every $N\in\mathbb{N}$ and $\xi\in[-1/2/1,2)$ let for convenience
\[
\mathfrak{m}_N(\xi)\coloneq\frac{1}{N}\sum_{n=1}^Ne\big(\xi\lfloor P(n)\psi(\alpha n)\rfloor\big)\text{,}
\] 
and let $\delta=\delta(\omega)>0$ and $\chi=\chi(\alpha,P,\psi)$ be such \eqref{majarcestcomplete} ,  \eqref{minarcestcompletedif} and \eqref{minarcestcompleteeasy} hold, i.e.: such that 
\begin{equation}\label{majarcestcompletenew}
\sup_{|\xi|\le N^{-d+\delta}}|\mathfrak{m}_N(\xi)-\mathcal{L}_N(\xi)|\lesssim N^{-\chi}\text{,}
\end{equation}
\begin{equation}\label{minarcestcompletenew}
\sup_{N^{-d+\delta}<|\xi|\le1/2}|\mathfrak{m}_N(\xi)|\lesssim N^{-\chi}
\quad\text{and}\quad
\sup_{N^{-d+\delta}<|\xi|\le1/2}|\mathcal{L}_N(\xi)|\lesssim N^{-\chi}\text{.}
\end{equation}
Using this approximation together with the estimate \eqref{rvarforZforapproximant} we may bound the first summand in \eqref{finalsplitlongonly} as follows
\begin{equation}\label{rvarfinal}
\begin{split}
\big\|V^r\big(A_{\lfloor 2^{k^{\tau}}\rfloor}f:\,k\in\mathbb{N}_0\big)\big\|_{\ell^p(\mathbb{Z})}=\big\|V^r\big(\mathcal{F}^{-1}_{\mathbb{Z}}\big[\mathfrak{m}_{\lfloor 2^{k^{\tau}}\rfloor}\mathcal{F}_{\mathbb{Z}}[f]\big]:\,k\in\mathbb{N}_0\big)\big\|_{\ell^p(\mathbb{Z})}&
\\
\le\big\|V^r\big(\mathcal{F}^{-1}_{\mathbb{Z}}\big[\mathcal{L}_{\lfloor 2^{k^{\tau}}\rfloor}\eta\mathcal{F}_{\mathbb{Z}}[f]\big]:\,k\in\mathbb{N}_0\big)\big\|_{\ell^p(\mathbb{Z})}+\big\|V^r\big(\mathcal{F}^{-1}_{\mathbb{Z}}\big[\big(\mathfrak{m}_{\lfloor 2^{k^{\tau}}\rfloor}-\mathcal{L}_{\lfloor 2^{k^{\tau}}\rfloor}\eta\big)\mathcal{F}_{\mathbb{Z}}[f]\big]:\,k\in\mathbb{N}_0\big)\big\|_{\ell^p(\mathbb{Z})}&
\\
\lesssim\big\|V^r\big(\mathcal{F}^{-1}_{\mathbb{Z}}\big[\mathcal{L}_{N}\eta\mathcal{F}_{\mathbb{Z}}[f]\big]:\,N\in\mathbb{N}\big)\big\|_{\ell^p(\mathbb{Z})}+\Big\|\Big(\sum_{k\in\mathbb{N}_0}\big|\mathcal{F}^{-1}_{\mathbb{Z}}\big[\big(\mathfrak{m}_{\lfloor 2^{k^{\tau}}\rfloor}-\mathcal{L}_{\lfloor 2^{k^{\tau}}\rfloor}\eta\big)\mathcal{F}_{\mathbb{Z}}[f]\big]\big|^r\Big)^{1/r}\Big\|_{\ell^p(\mathbb{Z})}&
\\
\lesssim\frac{r}{r-2}\|f\|_{\ell^p(\mathbb{Z})}+\Big\|\sum_{k\in\mathbb{N}_0}\big|\mathcal{F}^{-1}_{\mathbb{Z}}\big[\big(\mathfrak{m}_{\lfloor 2^{k^{\tau}}\rfloor}-\mathcal{L}_{\lfloor 2^{k^{\tau}}\rfloor}\eta\big)\mathcal{F}_{\mathbb{Z}}[f]\big]\big|\Big\|_{\ell^p(\mathbb{Z})}
\\
\le\frac{r}{r-2}\|f\|_{\ell^p(\mathbb{Z})}+\sum_{k\in\mathbb{N}_0}\big\|\mathcal{F}^{-1}_{\mathbb{Z}}\big[\big(\mathfrak{m}_{\lfloor 2^{k^{\tau}}\rfloor}-\mathcal{L}_{\lfloor 2^{k^{\tau}}\rfloor}\eta\big)\mathcal{F}_{\mathbb{Z}}[f]\big]\big\|_{\ell^p(\mathbb{Z})}&\text{,}
\end{split}
\end{equation}
and to conclude it suffices to bound the final summand. This will be carried out using interpolation and the estimates \eqref{majarcestcompletenew} and \eqref{minarcestcompletenew}.

More precisely, for $p=2$ and every $k\in\mathbb{N}_0$ we get by Plancherel that
\begin{equation}\label{l2first}
\begin{split}
\big\|\mathcal{F}^{-1}_{\mathbb{Z}}\big[\big(\mathfrak{m}_{\lfloor 2^{k^{\tau}}\rfloor}-\mathcal{L}_{\lfloor 2^{k^{\tau}}\rfloor}\eta\big)\mathcal{F}_{\mathbb{Z}}[f]\big]\big\|_{\ell^2(\mathbb{Z})}
\le\big\|\mathfrak{m}_{\lfloor 2^{k^{\tau}}\rfloor}-\mathcal{L}_{\lfloor 2^{k^{\tau}}\rfloor}\eta\big\|_{L^{\infty}(\mathbb{T})}\|f\|_{\ell^2(\mathbb{Z})}\text{.}
\end{split}
\end{equation}
We claim that 
\begin{equation}\label{comparisonmleta}
\big\|\mathfrak{m}_{\lfloor 2^{k^{\tau}}\rfloor}-\mathcal{L}_{\lfloor 2^{k^{\tau}}\rfloor}\eta\big\|_{L^{\infty}(\mathbb{T})}\lesssim 2^{-\chi k^{\tau}}\text{,}
\end{equation}
and using the trivial estimates   $\big\|\mathfrak{m}_{\lfloor 2^{k^{\tau}}\rfloor}\big\|_{L^{\infty}(\mathbb{T})}\le 1$ and $\big\|\mathcal{L}_{\lfloor 2^{k^{\tau}}\rfloor}\eta\big\|_{L^{\infty}(\mathbb{T})}\le 1$ it suffices to establish \eqref{comparisonmleta} for $k\gtrsim 1$. In particular, let us assume that $k\gtrsim1 $ is such that $\lfloor 2^{k^{\tau}}\rfloor^{-d+\delta}\le 1/32$. Then for every $|\xi|\le 1/32$ we have 
\begin{equation}\label{eqq1}
\big|\mathfrak{m}_{\lfloor 2^{k^{\tau}}\rfloor}(\xi)-\mathcal{L}_{\lfloor 2^{k^{\tau}}\rfloor}(\xi)\eta(\xi)\big|=\big|\mathfrak{m}_{\lfloor 2^{k^{\tau}}\rfloor}(\xi)-\mathcal{L}_{\lfloor 2^{k^{\tau}}\rfloor}(\xi)\big|\lesssim 2^{-\chi k^{\tau}}\text{,}
\end{equation}
by $
\sup_{|\xi|\le 1/2}|\mathfrak{m}_N(\xi)-\mathcal{L}_N(\xi)|\lesssim N^{-\chi}
$, which clearly holds by \eqref{majarcestcompletenew} and \eqref{minarcestcompletenew}. On the other hand, for every $\xi\in[-1/2,1/2)$ such that $|\xi|>1/32$ we get for $k\gtrsim 1$ that
\begin{equation}\label{eqq2}
\big|\mathfrak{m}_{\lfloor 2^{k^{\tau}}\rfloor}(\xi)-\mathcal{L}_{\lfloor 2^{k^{\tau}}\rfloor}(\xi)\eta(\xi)\big|\le \big|\mathfrak{m}_{\lfloor 2^{k^{\tau}}\rfloor}(\xi)\big|+\big|\mathcal{L}_{\lfloor 2^{k^{\tau}}\rfloor}(\xi)\big|\lesssim 2^{-\chi k^{\tau}}\text{,}
\end{equation}
where we have used \eqref{minarcestcompletenew} together with the fact that $\lfloor 2^{k^\tau}\rfloor^{-d+\delta}\le 1/32$ for $k\gtrsim1$.

By \eqref{eqq1} and \eqref{eqq2} we get \eqref{comparisonmleta}, and returning to  \eqref{l2first} we get
\begin{equation}\label{l2firstcomplete}
\begin{split}
\big\|\mathcal{F}^{-1}_{\mathbb{Z}}\big[\big(\mathfrak{m}_{\lfloor 2^{k^{\tau}}\rfloor}-\mathcal{L}_{\lfloor 2^{k^{\tau}}\rfloor}\eta\big)\mathcal{F}_{\mathbb{Z}}[f]\big]\big\|_{\ell^2(\mathbb{Z})}
\lesssim 2^{-\chi k^{\tau}}\|f\|_{\ell^2(\mathbb{Z})}\text{.}
\end{split}
\end{equation}
Using the bound above in \eqref{rvarfinal} with $p=2$ yields
\begin{equation}
\begin{split}
\big\|V^r\big(A_{\lfloor 2^{k^{\tau}}\rfloor}f:\,k\in\mathbb{N}_0\big)\big\|_{\ell^2(\mathbb{Z})}
\lesssim\frac{r}{r-2}\|f\|_{\ell^2(\mathbb{Z})}+\|f\|_{\ell^2(\mathbb{Z})}\sum_{k\in\mathbb{N}_0}2^{-\chi k^{\tau}}\lesssim\frac{r}{r-2} \|f\|_{\ell^2(\mathbb{Z})}&\text{,}
\end{split}
\end{equation}
which yields the result in the case $p=2$.

For $p\in(1,\infty)$ with $p\neq 2$, we fix $p_0\in(1,\infty)$ such that $p$ is between $p_0$ and $2$ and note that 
\begin{equation}\label{lp0}
\begin{split}
\big\|\mathcal{F}^{-1}_{\mathbb{Z}}\big[\big(\mathfrak{m}_{\lfloor 2^{k^{\tau}}\rfloor}-\mathcal{L}_{\lfloor 2^{k^{\tau}}\rfloor}\eta\big)\mathcal{F}_{\mathbb{Z}}[f]\big]\big\|_{\ell^{p_0}(\mathbb{Z})}
\le\big\|\mathcal{F}^{-1}_{\mathbb{Z}}\big[\mathfrak{m}_{\lfloor 2^{k^{\tau}}\rfloor}\mathcal{F}_{\mathbb{Z}}[f]\big]\big\|_{\ell^{p_0}(\mathbb{Z})}+\big\|\mathcal{F}^{-1}_{\mathbb{Z}}\big[\mathcal{L}_{\lfloor 2^{k^{\tau}}\rfloor}\eta\mathcal{F}_{\mathbb{Z}}[f]\big]\big\|_{\ell^{p_0}(\mathbb{Z})}
\\
\lesssim\big\|A_{\lfloor 2^{k^{\tau}}\rfloor}f\big\|_{\ell^{p_0}(\mathbb{Z})}+\|f\|_{\ell^{p_0}(\mathbb{Z})}\lesssim \|f\|_{\ell^{p_0}(\mathbb{Z})}
\text{,}
\end{split}
\end{equation}
where for the second summand, one may use an identical argument to the one used in \cite{frld}, see the beginning of the proof of Theorem~3 therein. Since the operators $T_k\coloneq \ell^q(\mathbb{Z})\to \ell^q(\mathbb{Z})$, $q\in(1,\infty)$, such that for every $f\in \ell^2(\mathbb{Z})\cap\ell^q(\mathbb{Z})$ we have 
\[
T_kf=\mathcal{F}^{-1}_{\mathbb{Z}}\big[\big(\mathfrak{m}_{\lfloor 2^{k^{\tau}}\rfloor}-\mathcal{L}_{\lfloor 2^{k^{\tau}}\rfloor}\eta\big)\mathcal{F}_{\mathbb{Z}}[f]\big]
\]
are linear and satisfy the bounds $\|T_k\|_{\ell^2(\mathbb{Z})\to\ell^2(\mathbb{Z})}\lesssim 2^{-\chi k^{\tau}}$ and $\|T_k\|_{\ell^{p_0}(\mathbb{Z})\to\ell^{p_0}(\mathbb{Z})}\lesssim 1$, we may apply Riesz--Thorin interpolation between $p_0$ and $2$ to conclude that there exists a positive constant $\chi'=\chi'(p,p_0)$ such that
\[
\|T_k\|_{\ell^p(\mathbb{Z})\to\ell^p(\mathbb{Z})}\lesssim 2^{-\chi' k^{\tau}}\text{.}
\]
Using this in \eqref{rvarfinal} yields
\begin{equation}
\begin{split}
\big\|V^r\big(A_{\lfloor 2^{k^{\tau}}\rfloor}f:\,k\in\mathbb{N}_0\big)\big\|_{\ell^p(\mathbb{Z})}
\lesssim\frac{r}{r-2}\|f\|_{\ell^p(\mathbb{Z})}+\|f\|_{\ell^p(\mathbb{Z})}\sum_{k\in\mathbb{N}_0}2^{-\chi' k^{\tau}}\lesssim\frac{r}{r-2} \|f\|_{\ell^p(\mathbb{Z})}&\text{,}
\end{split}
\end{equation}
and by \eqref{finalsplitlongonly} the proof is complete.
\end{proof}

\end{document}